\documentclass[a4paper,11pt]{article}
\usepackage[latin1]{inputenc}
\usepackage[T1]{fontenc}
\usepackage[british,UKenglish,USenglish,american]{babel}
\usepackage{amsmath}
\usepackage{tikz}
\usetikzlibrary{decorations.markings}
\usepackage{mathtools}
\usepackage{amsfonts}
\usepackage{hyperref}
\usepackage{enumitem}
\usepackage{hyperref}
\usepackage{comment}
\usepackage[T1]{fontenc}
\usepackage[latin1]{inputenc}
\usepackage{ntheorem}
\usetikzlibrary{arrows.meta, positioning}
\usepackage{theorem}
\usepackage[normalem]{ulem}
\usepackage{graphics}
\usepackage{makeidx}
\usepackage{fancyhdr}
\usepackage{hyperref}
\usepackage{tikz}
\usepackage{todonotes}
\usepackage{verbatim}
\usepackage{bbm}
\usepackage{bm}
\usepackage{amssymb}
\usepackage{enumitem}
\usepackage{color}
\usepackage{relsize}
\numberwithin{equation}{section}
\usepackage{amssymb}
\newcommand{\be}{\begin{equation}}
	\newcommand{\ee}{\end{equation}}
\newcommand{\benn}{\begin{equation*}}
	\newcommand{\eenn}{\end{equation*}}
\newcommand{\bea}{\begin{eqnarray}}
	\newcommand{\eea}{\end{eqnarray}}

\newcommand{\beann}{\begin{eqnarray*}}
	\newcommand{\eeann}{\end{eqnarray*}}

\newtheorem{theorem}{Theorem}[section]
\newtheorem{proposition}[theorem]{Proposition}
\newtheorem{corollary}[theorem]{Corollary}
\newtheorem{lemma}[theorem]{Lemma}

\newtheorem{definition}[theorem]{Definition}
\theorembodyfont{\rm}
\newtheorem{remark}[theorem]{Remark}
\newtheorem{notation}[theorem]{Notation}

\newtheorem{example}[theorem]{Example}
\newtheorem{assumptions}[theorem]{Assumption}

\newcommand{\proof}{\noindent{\bf Proof.\, }}
\newcommand{\qed}{\hfill $\Box$\smallskip\newline}
\newcommand{\E}{\noindent{$\mathbb{E}$ \ }}

\def\R{\mathbb{R}}

\def\N{\mathbb{N}}
\def\Z{\mathbb{Z}}
\def\P{\mathbb{P}}
\def\E{\mathbb{E}}
\def\T{\mathbb{T}}

\def\P{\mathbb{P}}
\def\cA{\mathcal{A}}
\def\cB{\mathcal{B}}

\def\cF{\mathcal{F}}

\def\cH{\mathcal{H}}

\def\cM{\mathcal{M}}

\def\cQ{\mathcal{Q}}
\def\cR{\mathcal{R}}

\def\cX{\mathcal{X}}

\def\txtd{{\textnormal{d}}}

\allowdisplaybreaks
\newcommand*\samethanks[1][\value{footnote}]{\footnotemark[#1]}

\title{Synchronization by noise for stochastic differential equations driven by fractional Brownian motion}
\author{Alexandra Blessing Neam\c tu~\thanks{  Department of Mathematics and Statistics, University of Konstanz,
		Universit\"atsstra\ss{}e~10, 78464 Konstanz, Germany. E-Mail: alexandra.blessing@uni-konstanz.de,  mazyar.ghani-varzaneh@uni-konstanz.de}~~~~and~~~~Mazyar Ghani Varzaneh \samethanks}
\begin{document}
	\maketitle
	\begin{abstract}
		We investigate synchronization by noise for stochastic differential equations (SDEs) driven by a fractional Brownian motion (fbm) with Hurst index $H\in(0,1)$. Provided that the SDE has a negative top Lyapunov exponent, we show that a weak form of synchronization occurs. To this aim we use tools from stochastic dynamical systems, random dynamical systems and  characterize the support of an invariant measure of a random dynamical system in a non-Markovian setting.
		
	\end{abstract}
	\tableofcontents
	\section{Introduction} 
	The main goal of this work is to prove synchronization by noise for stochastic differential equations (SDEs) driven by fractional Brownian motion (fbm). In particular, we consider SDEs with additive noise
	\begin{align}\label{sde:intro}
		\begin{cases}
			& \txtd Y_ t = F(Y_t)~\txtd t + \sigma\txtd B^H_t\\
			&Y_0=x\in \R^d,
		\end{cases}
	\end{align}
	where $d\geq 1$, $(B^H_t)_{t\geq 0}$ is a $d$-dimensional fractional Brownian motion with Hurst index $H\in(0,1)$, the drift term $F$ satisfies a suitable dissipativity and monotonicity assumption and $\sigma$ denotes the intensity of the noise.\\
	
	Synchronization means that trajectories of \eqref{sde:intro} starting from different initial data converge in a certain sense to a single trajectory given the same realization of the noise.
	In the language of random dynamical systems, this property can be viewed as convergence towards a unique random equilibrium point (called singleton attractor). In the Markovian case, a powerful tool for the analysis of the long-time behavior of \eqref{sde:intro} is to associate to the random dynamical system a Markov semigroup and relate its invariant measure to a statistical equilibrium, see \cite{CF94} and Remark \ref{RESFSS}. This is one of the main challenges which one has to overcome in the non-Markovian setting.  \\
	
	Referring to \cite{FGS16a}, synchronization occurs if there exists a weak attractor (as introduced in Subsection \ref{sec:p:attractors}) for the random dynamical system generated by \eqref{sde:intro} consisting of a single point.~As argued in \cite{FGS16a}, it is meaningful to weaken this concept and impose the existence of a minimal weak point attractor consisting of a single point. This is called weak synchronization and can be verified under natural assumptions on the underlying dynamical system established in \cite{FGS16a,FGS17}. In the Markovian case, if the associated Markov semigroup is strongly mixing, there is a connection between its invariant measure and the weak point attractor \cite{FGS16a} which is not available in our setting.  \\
	
	In this work, we also focus on a modified concept of weak synchronization tailored to the non-Markovian dynamics induced by \eqref{sde:intro}. To this aim, we follow the recent approach developed in \cite{BNGH26} that
	constructs a random dynamical system (RDS) given a stochastic dynamical system (SDS) as introduced in~\cite{Hai05}.~In this framework, \cite{BNGH26} established that \eqref{sde:intro} has a negative top Lyapunov exponent when increasing the noise intensity $\sigma$. Here we go one step further and prove that the trajectories of \eqref{sde:intro} converge in a certain sense to a random fixed point independently of the initial datum.~To the best of our knowledge, this is the first result on synchronization by noise for SDEs driven by fbm in dimensions higher than one, i.e.~beyond the order-preserving case investigated in~\cite[Subsection 4.1]{FGS17}.~Notably, our concept of weak synchronization is natural in a non-Markovian setting but is slightly different than the one defined in \cite{FGS16a,FGS17}.~As already stated, this means that the minimal weak point attractor of the RDS associated to~\eqref{sde:intro} is a singleton. \\

	In the following, we briefly explain our concept of synchronization and its connection to \cite{FGS16a,FGS17}. To this aim, we first recall that a major step for the analysis of non-Markovian dynamics was introduced in~\cite{Hai05}, where it was shown that enlarging the state space incorporating the noise, one can obtain a Markov process on the extended space. Using tools from Markov processes, one can establish the existence of a unique invariant measure on this extended space. We refer to \cite{Hai05, HO07, HP11, HP13} for more details on this procedure and further results. Under certain assumptions, the projection of this measure onto the original state space is mixing, which was used in \cite{FGS16a} to prove synchronization for one-dimensional SDEs driven by fbm. Moreover, a key feature of the projected measure is that it admits a density with respect to Lebesgue measure and satisfies Gaussian-type upper and lower bounds \cite{LPS23}. Such bounds were used in \cite{BNGH26} to prove that the top Lyapunov exponent of \eqref{sde:intro} is negative. \\
	
	In this work, we follow a complementary approach. We do not consider the projected measure but instead we disintegrate it with respect to the Wiener measure. This leads to a family of random measures $\{\mu_\omega\}_{\omega\in\Omega}$, (where $(\Omega,\cF,\P)$ is a suitable probability space) on the original state space which are shown to be invariant under the underlying dynamics.~These random measures can be viewed as analogues of statistical equilibria in the Markovian setting.~When the top Lyapunov exponent of~\eqref{sde:intro} is negative, we show that these measures are discrete similar to the Markovian situation discussed in \cite{FGS16a}. For a strongly mixing, white-noise RDS, it is known that the family $\{\mu_\omega\}_{\omega\in\Omega}$ constitutes a weak point attractor as established in \cite{FGS16a}. In the absence of the Markov property, we cannot prove an analogous result.~However we are able to show an averaged weak attraction in Proposition~\ref{POINT} which seems to be natural in a non-Markovian setting.~Furthermore, we show that the support of these measures is contained in the weak point attractor.~This motivates us to introduce the following formal notion of synchronization, see Definition \ref{DEF:WEAK} for the precise statement. 
	\begin{definition}\label{def:intro}
		We say that weak synchronization occurs if there exists a random point $b:\Omega\to \R^d$ such that 
		\[ \text{\em supp}(\mu_\omega) = \{ b(\omega) \}, \hspace*{1 mm} \P-a.e. \hspace*{1 mm}  \omega\in\Omega. \]
	\end{definition}

	\paragraph{Main result}
	At an informal level, our main result Theorem  \ref{WESYN} reads as follows.
	
	\begin{theorem}\label{main:info}
		The SDE \eqref{sde:intro} exhibits weak synchronization in the sense of Definition \ref{def:intro} provided that $\sigma$ is large enough and $F$ satisfies Assumption \ref{AASQw98a}.
	\end{theorem}
	
	In this setting we are able to show that \eqref{sde:intro} has a negative top Lyapunov exponent as argued in \cite{BNGH26}. The optimality of the assumptions on $F$ and $\sigma$ ensuring the negativity of the top Lyapunov exponent will be investigated in a future work, in particular an interplay between $F$ and $\sigma$ leading to synchronization, as well as the speed of convergence towards the equilibrium depending on the Hurst parameter $H\in(0,1)$. \\
	
	The proof of Theorem~\ref{main:info} relies on the structure of the disintegrated measures together with the assumption on the drift term which ensures a contraction on large sets.~In particular, the negativity of the top Lyapunov exponent implies that these measures are atomic, as justified in Proposition~\ref{descrte}.~This fact together with the assumption on the drift ensures that the distance between the points in their support is bounded by a deterministic constant, as summarized in Corollary~\ref{ATRTASs}. Furthermore, we impose an assumption on the drift which implies that trajectories of~\eqref{sde:intro} synchronize outside large balls. 
	In order to control the possible expansion within bounded sets, we impose a natural condition requiring that the total mass of the projected invariant measure of~\eqref{sde:intro} in $\mathbb{R}^d$ is large on the complements of these sets. 
	This condition can be verified increasing the intensity of the noise. 
	
	\paragraph{Literature} 
	
	The phenomenon of synchronization (by noise) was investigated for numerous stochastic systems in finite and infinite-dimensions ranging from SDEs on compact manifolds \cite{Baxendale} to singular SPDEs~\cite{GT20,Tommaso}.~A well-known framework for the investigation of synchronization by noise is given by order-preserving, strongly mixing RDS.~This setting has been analyzed e.g.~in~\cite{DP,AC,C,CS,FGS17} and the references specified therein.~For more details on synchronization and an overview of the available results we refer to~\cite{BS}.~We mention that order-preservation holds under restrictive conditions on the drift term in dimensions higher than one.~The tools developed in this manuscript are applicable beyond order-preserving white-noise RDS.~Notably, as already mentioned, synchronization attracted lots of interest also in the context of singular stochastic partial differential equations.~For example~synchronization by noise for the $\Phi^4_2$ and $\Phi^4_3$ equation on the torus was derived in \cite{GT20}.~Moreover, synchronization for the KPZ equation on the torus (referred to as one-force one-solution principle) was proved in \cite{Tommaso}.~We also mention other approaches to synchronization by noise based on large deviations which consider a small noise regime~\cite{T,BeGe:13,BW:17}.\\

	Finally, we point out that synchronization can provide a pathwise coupling mechanism used to establish ergodicity.~However, even for Markovian systems, proving ergodicity only via synchronization is technically challenging. Consequently, many works establish ergodicity using alternative coupling methods \cite{MHS11, E16, KSH18}.~The non-Markovian nature of the systems leads to additional technical difficulties~\cite{Hai05}.\\

Moreover, there has been a growing interest in the dynamics of stochastic systems with fractional Brownian motion or other types of non-Markovian noise. In particular, stationary solutions for SDEs driven by fbm together with Gaussian bounds for their densities have been established in \cite{LPS23} whereas rates of convergence of fractional SDEs towards their equilibria and related aspects have been investigated in  \cite{CP11, CPT14, FP17, DPT19, PTV20}.~Furthermore, several results on fractional stochastic dynamics such as averaging~\cite{averaging}, sample path estimates~\cite{Katharina,Nils}, Lyapunov exponents \cite{BNGH26}, finite-time Lyapunov exponents~\cite{DB}, amplitude equations \cite{AE}, early warning signs for bifurcations~\cite{EWS} or traveling waves~\cite{Stannat} have been established. For results on invariant manifolds and stability, we refer to~\cite{KN23, GVR25, GVR26} and to \cite{GVR25B} for delay equations.~Ergodicity for singular SDEs with fractional Brownian motion was recently proved in \cite{Avi}.~Motivated by these recent developments, we constructed in \cite{BNGH26} random dynamical systems given stochastic dynamical systems generated by SDEs with additive fractional Brownian motion.~This framework enabled us to show in a certain regime that the top Lyapunov exponent of such an SDE is negative.~This assumption is heavily used in this work to show synchronization.~The complementary regime in which the top Lyapunov exponent is positive, is expected to lead to chaotic behavior of the systems together with random measures exhibiting SRB-type properties.~This will be analyzed in a forthcoming work.  \\

\paragraph{Structure of the paper} 
In Section \ref{sec:p} we collect basic concepts from stochastic dynamical systems introduced in \cite{Hai05}, invariant measures, fractional Brownian motion, random dynamical systems and attractors which are used to define (weak) synchronization.~Moreover, we state some results established in~\cite{BNGH26} on the generation of an RDS given an SDS.~In Section \ref{sec:3} we thoroughly discuss the concepts of invariant measures and attractors for RDS. A particular interest is given by the support of the disintegrated invariant measure of a stochastic dynamical system constructed in Lemma \ref{RAND_MEASURE} provided that the top Lyapunov exponent of the underlying system is negative. For Markovian systems, it is known that the corresponding measure is atomic, see \cite[Lemma 2.19]{FGS17}. We establish in Proposition \ref{descrte} an analogue result using a stable manifold theorem. Furthermore, we introduce in Subsection \ref{RAAT} a concept of a random attractor for non-Markovian systems, which we call P-attractor. This depends only on the past of the noise and is motivated by the theory of SDS which filter out the future of the noise. In Theorem \ref{ATTRRP} we provide the existence of such an attractor for the RDS generated by \eqref{sde:intro}. 
Section \ref{sec:main} contains the main result on synchronization by noise for \eqref{sde:intro}. Since in the Markovian case, there exists a one-to-one correspondence between invariant measures of RDS and pullback attractors \cite{CF94}, we investigate such questions for the disintegrated measure of the SDS and the P-attractor. 
Therefore, Subsection \ref{SUPPAS} provides a connection between the support of the invariant measure of the RDS and the introduced weak point attractor. Namely, we show that the support of this measure is contained in the P-attractor, see Proposition \ref{YASHAsd} for a precise statement together with an averaged weak attraction in Proposition~\ref{POINT}.~Based on these insights, Theorem \ref{WESYN} establishes synchronization by fractional noise. 
As already mentioned, one essential ingredient for this argument is that the support of the invariant measure of the RDS contains finitely many discrete points and their distance is always bounded from above by a deterministic constant.~This holds due to the assumptions imposed on the drift, particularly a contraction property on large sets, and due to the additive structure of the noise. We emphasize that the tools for synchronization developed in this work are not restricted to order-preserving random dynamical systems.~Section \ref{sec:ex} provides an example of an SDE for which weak synchronization holds and Section \ref{outlook} points out some related open problems which will be considered in a forthcoming work.~We conclude with Appendix \ref{appendix}, which contains a result from measure theory required for the proof of Corollary \ref{STABSS}.

\subsection*{Acknowledgements}
\label{sec:acknowledgements}
The authors acknowledge funding by the Deutsche Forschungsgemeinschaft (DFG, German Research Foundation) - CRC/TRR 388 "Rough Analysis, Stochastic Dynamics and Related Fields" - Project ID 516748464.

\section{Preliminaries}\label{sec:p}
\subsection{Notations}
In this section, we introduce some notations, definitions and auxiliary results which are frequently used throughout the manuscript.
\begin{lemma}\label{pushforward}
	Let \((X, \sigma(X))\) and \((Y, \sigma(Y))\) be two measurable spaces and let \(T: X \to Y\) be a measurable mapping.  
	Suppose that \(\tilde{\mu}\) is a Borel measure defined on \((X, \sigma(X))\).  
	The pushforward measure of \(\tilde{\mu}\) under \(T\), denoted by \(T_{\ast}\tilde{\mu}\), is defined on \(Y\) by
	\[
	T_{\ast}\tilde{\mu}(E) := \tilde{\mu}(T^{-1}(E)), \qquad \forall E \in \sigma(Y).
	\]
	Furthermore, for every measurable function \(f: Y \to \mathbb{R}^{+}\), the following identity holds
	\[
	\int_{X} f(T(x)) \, \tilde{\mu}(\mathrm{d}x) 
	= \int_{Y} f(y) \, T_{\ast}\tilde{\mu}(\mathrm{d}y).
	\]
\end{lemma}
\begin{definition}
	Let $(\Omega,\mathcal{F},\mathbb{P})$ be a probability space. 
	The space of random Lipschitz functions denoted by \(\mathrm{BL}_{\Omega}(\mathbb{R}^d) \) consists of all functions \( g : \mathbb{R}^d \times \Omega \rightarrow \mathbb{R} \) satisfying the following properties.
	\begin{itemize}
		\item For every \( \omega \in \Omega \), the function \( x \mapsto g(x,\omega) \) is continuous.
		\item For every \( x \in \R^d \), the function \( \omega \mapsto g(x,\omega) \) is measurable.
		\item There exists a constant \( M > 0 \) such that on a set of full measure in \( \Omega \) 
		\begin{align*}
			\| g(\cdot, \omega) \|_{\mathrm{BL}} < M,
		\end{align*}
		where
		\[
		\| g(\cdot, \omega) \|_{\mathrm{BL}} := \sup_{x \in \mathbb{R}^d} |g(x,\omega)| + \sup_{\substack{x, y \in \mathbb{R}^d \\ x \neq y}} \frac{|g(x,\omega) - g(y,\omega)|}{|x - y|}.
		\]
	\end{itemize}
	Furthermore 
	\[
	\mathrm{Pr}_{\Omega}(\mathbb{R}^d) := \left\{ \nu \text{ is a probability measure on } \mathbb{R}^d \times \Omega \,\middle|\, \left( \Pi_{\Omega} \right)_{\ast} \nu = \mathbb{P} \right\},
	\]
	where $\Pi$ denotes the projection onto $\Omega$. 
\end{definition}
\subsection{Random dynamical systems}\label{sec:p:attractors}
We begin by recalling the basic definitions of the theory of random dynamical systems~\cite{Arn98}, which is largely inspired by the framework of deterministic dynamical systems and ergodic theory.
\begin{definition}\label{SAGBAsd}
	Let \((\Omega, \mathcal{F})\) be a measurable space, and let \(\T\) denote either \(\mathbb{R}\) or \(\mathbb{R}^+\), each equipped with its Borel \(\sigma\)-algebra.  
	A family of mappings \(\theta = \{\theta_t\}_{t \in \T}\) on \(\Omega\) is called a \emph{measurable dynamical system} if the following conditions hold.
	\begin{itemize}
		\item The mapping \((t, \omega) \mapsto \theta_t(\omega)\) is measurable with respect to \(\sigma(\T) \otimes \mathcal{F}\) and \(\mathcal{F}\).
		\item The identity condition is satisfied at \(t=0\), i.e., \(\theta_0 = \mathrm{Id}\).
		\item The semigroup property holds: \(\theta_{s+t} = \theta_s \circ \theta_t\) for all \(s, t \in \T\).
	\end{itemize}
	If there exists a probability measure \(\mathbb{P}\) on \((\Omega, \mathcal{F})\) such that
	\[
	(\theta_t)_{\ast}\mathbb{P} = \mathbb{P}, \quad \forall\, t \in \T,
	\]
	then the quadruple \((\Omega, \mathcal{F}, \mathbb{P}, \theta)\) is called a \emph{measurable metric dynamical system}.  
	Furthermore, \((\Omega, \mathcal{F}, \mathbb{P}, \theta)\) is said to be {ergodic}, if for every \(t \in \T \setminus \{0\}\), the transformation \(\theta_t\) is ergodic, meaning that for any \(A \in \mathcal{F}\),
	\[
	\mathbb{P}\big(A \triangle \theta_t^{-1}(A)\big) = 0
	\quad \Rightarrow \quad 
	\mathbb{P}(A) \in \{0,1\}.
	\]
\end{definition}
Now, we can define the concept of a random dynamical system.
\begin{definition}\label{Olas5d}
	Let \(\mathcal{X}\) be a Polish space equipped with its Borel \(\sigma\)-algebra \(\sigma(\mathcal{X})\).  
	An {(ergodic) measurable random dynamical system} on \((\mathcal{X}, \sigma(\mathcal{X}))\) consists of an (ergodic) measurable metric dynamical system \((\Omega, \mathcal{F}, \mathbb{P}, \theta)\) together with a measurable function
	\[
	\Phi : \mathbb{R}^+ \times \mathcal{X} \times \Omega \to \mathcal{X},
	\]
	satisfying the {cocycle property}.  
	More specifically, for each \((t, x, \omega) \in \mathbb{R}^+ \times \mathcal{X} \times \Omega\), set \(\Phi^t_\omega(x) := \Phi(t,x,\omega)\).  
	Then the following conditions hold:
	\[
	\Phi^0_\omega = \operatorname{Id}_{\mathcal{X}} \quad \text{for all } \omega \in \Omega,
	\]
	and
	\[
	\Phi^{t+s}_\omega(x) = \Phi^s_{\theta_t \omega}\big(\Phi^t_\omega(x)\big),
	\]
	for all \(s, t \in \mathbb{R}^+\), \(x \in \mathcal{X}\), and \(\omega \in \Omega\).  
	The function \(\Phi\) is called a \emph{cocycle}.  
	Given the condition that, for each \(\omega \in \Omega\), the map \((t, x) \mapsto \Phi(t, x, \omega)\) is continuous, the function \(\Phi\) is called a {continuous random dynamical system}.  
	Moreover, for \(k \geq 1\), we call \(\Phi\) a {\(C^k\)-random dynamical system} if for every \(t \geq 0\) and \(\omega \in \Omega\), the function
	\[
	x \mapsto \Phi^t_\omega(x)
	\]
	is \(C^k\)-Fr\'{e}chet differentiable.
\end{definition}
\begin{definition}\label{ARTASS}
	Let \((\Omega, \mathcal{F}, \mathbb{P}, \theta)\) be a measurable metric dynamical system, and let \(\mathcal{X}\) be a Polish space.
	\begin{itemize}
		\item A random variable \(Y : \Omega \rightarrow (0, \infty)\) is called \emph{tempered} if
		\begin{align*}
			\limsup_{|t| \to \infty,\, t \in \mathbb{T}} \frac{1}{|t|} \log^{+} Y(\theta_t \omega) = 0,
			\quad \text{for } \mathbb{P}\text{-almost every } \omega \in \Omega,
		\end{align*}
		where \(\log^{+} x := \max\{ \log x, 0 \}\).
		
		\item Let \(\mathcal{E}=\{\mathcal{E}(\omega)\}_{\omega \in \Omega}\) be a family of non-empty subsets of \(\mathcal{X}\) such that, for every \(x \in \mathcal{X}\), the function \(\omega \mapsto d(x, \mathcal{E}(\omega))\) is measurable, and with probability one each \(\mathcal{E}(\omega)\) is bounded in \(\mathcal{X}\). We call \(\mathcal{E}\) a {tempered set} if the random variable
		\[
		r(\mathcal{E}(\omega)) := \sup_{x \in \mathcal{E}(\omega)} d(x, x_0)
		\]
		is tempered for some fixed reference point \(x_0 \in \mathcal{X}\). The collection of tempered sets in $\mathcal{X}$ is denoted by $\mathcal{D}(\mathcal{X})$.
	\end{itemize}
\end{definition}

\begin{remark}\label{TEMOER}
	According to \cite[Theorem~4.1.3]{Arn98}, the random variable \(Y : \Omega \rightarrow (0, \infty)\) is tempered if
	\begin{align*}
		\mathbb{E}\left[\sup_{t \in \mathbb{T} \cap [0,1]} Y(\theta_t \omega)\right] < \infty.
	\end{align*}
\end{remark}
An important topic in the theory of random dynamical systems is the concept of pullback attractors. Before stating it, we need the following definition. 
\begin{definition}
	Let $\Phi$ be an RDS and \( \mathcal{A}=\{\mathcal{A}(\omega)\}_{\omega \in \Omega} \) be a family of non-empty subsets of  $\mathcal{X}$. 
	\begin{enumerate}
		\item  [1)] The family \( \mathcal{A} \) is called a random closed set (random compact set) if, with probability one, each \( \mathcal{A}(\omega) \) is closed (compact), and for every \( x \in \mathcal{X} \), the function \( \omega \mapsto d(x, \mathcal{A}(\omega)) \) 
		is \( \mathcal{F} \)-measurable.
		\item [2)] We say that $\mathcal{A}$ is invariant under \( \Phi \) if for all \( t \geq 0 \),
		\[
		\Phi^{t}_{\omega}(\mathcal{A}(\omega)) = \mathcal{A}(\theta_t\omega) 
		\]
		holds for \( \mathbb{P} \)-almost every \( \omega \in \Omega \).
		\item [3)] The set $\mathcal{A}$ is called {random absorbing} if, 
		for every $\mathcal{E} \in \mathcal{D}(\mathcal{X})$, there exists a (random) time 
		$T = T(\mathcal{E}, \omega)$ such that
		\[
		\varphi(t, \theta_{-t}\omega, \mathcal{E}(\theta_{-t}\omega)) \subset \mathcal{A}(\omega),
		\quad \forall\, t \ge T(\mathcal{E}, \omega),
		\]
		for $\mathbb{P}$-almost every $\omega \in \Omega$.
		
	\end{enumerate}
\end{definition} 
Now we can give the definition of a pullback attractor.
\begin{definition}\label{ATTRACTOR}
	Let $\Phi$ be an RDS. A random compact set $\cA$ is { called pullback attractor} if 
	\begin{enumerate}
		\item [1)] $\cA$ is invariant under \( \Phi \).
		\item [2)] For every compact set $K\subseteq  \mathcal{X}$ 
		we have
		\[ \lim\limits_{t\to\infty} \sup\limits_{x\in K} d(\Phi^t_{\theta_{-t}\omega}(x), \cA(\omega)  )=0~~~\P-a.s. \]
	\end{enumerate}
	We call $\cA$ a {weak attractor} if the convergence in 2) holds in probability instead of almost surely. The set $\mathcal{A}$ is called a {(weak) point attractor} if it satisfies the properties stated above, assuming that the compact set $K$ in condition 2) consists of a single point.~A (weak) point attractor is said to be {minimal} if it is contained in every other (weak) point attractor.	%
\end{definition}
\begin{remark}
	\begin{itemize}
		\item Obviously, every pullback attractor is a weak attractor but the converse is not true, see \cite{S:Attr}.
		\item In \cite[Definition~3.9]{CF94}, the second property in Definition~\ref{ATTRACTOR} is stated under the stronger assumption that $K$ is bounded. When $\mathcal{X}$ is a finite-dimensional Banach space, these two concepts are  equivalent.
	\end{itemize}
\end{remark}
For more details on the construction and properties of random attractors and related concepts, in particular where the set $K$ above is replaced by a random bounded or compact set, we refer to \cite{CF94, FS96, Gess1, Gess2, Andre}.~Finally, we point out the following well-known fact for the weak point attractor of a white-noise RDS in the Markovian case.    

\subsection{Stochastic dynamical systems}
We now introduce the basic definitions of the theory of stochastic dynamical systems developed in \cite{Hai05}, which has a mostly probabilistic flavor.
\begin{definition}\label{SDS}
	Let $\mathcal{B}$ be a Polish space equipped with the Borel $\sigma$-algebra $\sigma(\mathcal{B})$. 
	A quadruple $\left(\mathcal{B}, \{P_t\}_{t \geq 0}, \mathbf{P}, \{\vartheta_t\}_{t \geq 0}\right)$ is called a {stationary noise process} if the following properties hold.
	\begin{itemize}
		\item [(i)]\label{I} The measure $\mathbf{P}$ is a probability measure on $(\mathcal{B}, \sigma(\mathcal{B}))$. Also, for each $t \geq 0$, the maps 
		$P_t: \mathcal{B} \times \mathcal{B} \to \mathcal{B}$ and 
		$\vartheta_t: \mathcal{B} \to \mathcal{B}$ 
		are measurable, and for every $\omega^- \in \mathcal{B}$, the random variables 
		$P_t(\omega^-, \cdot)$ and $\vartheta_t$ are independent.
		\item [(ii)]\label{II} The family $\{\vartheta_t\}_{t \geq 0}$ forms a semiflow on $\mathcal{B}$, i.e. 
		$\vartheta_{t+s} = \vartheta_t \circ \vartheta_s$, and preserves $\mathbf{P}$, that is $(\vartheta_t)_{\star}\mathbf{P} = \mathbf{P}$.
		\item [(iii)]\label{III} For $\omega^- \in \mathcal{B}$ and $t \geq 0$, define
		\[
		\mathcal{P}_t(\omega^-; \cdot) := (P_t(\omega^-, \cdot))_{\star}(\mathbf{P}).
		\]
		The family $\{\mathcal{P}_t\}_{t \geq 0}$ forms a Feller transition semigroup on $\mathcal{B}$.
		Moreover, $\mathbf{P}$ is the unique invariant measure of $\{\mathcal{P}_t\}_{t \geq 0}$.
		\item [(iv)]\label{IV} For all $\omega^-, \omega^+ \in \mathcal{B}$ and $t,s \geq 0$,
		\begin{align*}
			&\vartheta_t \circ P_t(\omega^-, \omega^+) 
			= P_t(\vartheta_t \omega^-, P_t(\omega^+, \omega^-)) = \omega^-,\\
			&P_{t+s}(\omega^-, \omega^+) 
			= P_t(P_s(\omega^-, \omega^+), \vartheta_s \omega^+).
		\end{align*}
	\end{itemize}
\end{definition}
This leads to the following definition.
\begin{definition}\label{SDSSS}
	Let $\left(\mathcal{B}, \{P_t\}_{t \geq 0}, \mathbf{P}, \{\vartheta_t\}_{t \geq 0}\right)$ be a stationary noise process, and let $\mathcal{X}$ be a Polish space. 
	A map
	\begin{align}
		\begin{split}
			&\varphi: \mathbb{R}^+ \times \mathcal{X} \times \mathcal{B} \to \mathcal{X},\\
			&(t, x, \omega) \mapsto \varphi^t_{\omega}(x),
		\end{split}
	\end{align}
	is called a {continuous stochastic dynamical system} if
	\begin{itemize}
		\item For every $T > 0$, $\omega \in \mathcal{B}$, and $x \in \mathcal{X}$,
		\[
		\tilde{\varphi}_T(x,\omega)(t) := \varphi^{t}_{\vartheta_{T - t}\omega}(x), \quad t \in [0, T],
		\]
		lies in $C([0, T], \mathcal{X})$, and $(x,\omega) \mapsto \tilde{\varphi}_T(x,\omega)$ 
		is continuous from $\mathcal{X} \times \mathcal{B}$ to $C([0, T], \mathcal{X})$.
		
		\item For all $t,s \geq 0$, $\omega \in \mathcal{B}$, and $x \in \mathcal{X}$,
		\begin{align}\label{cocycle_SDS}
			\begin{split}
				&\varphi^{0}_{\omega}(x) = x,\\
				&\varphi^{t+s}_{\omega}(x)
				= \varphi^{s}_{\omega}\big(\varphi^{t}_{\vartheta_{s}\omega}(x)\big).
			\end{split}
		\end{align}
		
		\item If $\mathcal{X}$ is a separable Banach space, the SDS is called $C^k$ if for all $k \geq 1$, $t \geq 0$, and $\omega \in \mathcal{B}$, 
		the map $x \mapsto \varphi^t_{\omega}(x)$ is $C^k$-Fr\'echet differentiable.
	\end{itemize}
\end{definition}
We now focus on the invariant measure of a continuous SDS. The following definition is taken from \cite{Hai05}.
\begin{definition}\label{DEAFRS}
	Let $\varphi$ be a continuous SDS. For $t \geq 0$ and $\Gamma \in \sigma(\mathcal{X}) \otimes \sigma(\mathcal{B})$, define
	\begin{align}\label{Chapman}
		\mathcal{Q}_{t}\left((x,\omega^-); \Gamma\right) 
		:= \int_{\mathcal{B}} 
		\chi_{\Gamma}\!\left( 
		\varphi^{t}_{P_{t}(\omega^-, \omega^+)}(x), 
		P_{t}(\omega^-, \omega^+) 
		\right) 
		\mathbf{P}(\mathrm{d}\omega^+).
	\end{align}
	The family $\{\mathcal{Q}_t\}_{t \geq 0}$ forms a Feller transition semigroup on $\mathcal{X} \times \mathcal{B}$.  
	
	A probability measure $\mu$ on $\mathcal{X} \times \mathcal{B}$ is called an \emph{invariant measure} for $\varphi$ if it satisfies:
	\begin{itemize}
		\item The marginal of $\mu$ on $\mathcal{B}$ coincides with $\mathbf{P}$, i.e.
		\begin{align}\label{UJas}
			(\Pi_\mathcal{B})_{\star}\mu = \mathbf{P}.
		\end{align}
		
		\item The measure $\mu$ is invariant under the semigroup $\{\mathcal{Q}_t\}_{t \geq 0}$, meaning that for all $t \geq 0$ and $\Gamma \in \sigma(\mathcal{X}) \otimes \sigma(\mathcal{B})$,
		\[
		(\mathcal{Q}_t)_{\star}\mu(\Gamma)
		:= \int_{\mathcal{X} \times \mathcal{B}}
		\mathcal{Q}_t\big((x,\omega); \Gamma\big)\,
		\mu(\mathrm{d}(x,\omega))
		= \mu(\Gamma).
		\]
	\end{itemize}
	We also call $\mu$ \emph{ergodic} if it is ergodic with respect to the shift map on 
	$C(\mathbb{R}^+, \mathcal{X} \times \mathcal{B})$; see \emph{\cite[Chapter~2]{DPZ96}} for further details.
\end{definition}
In the next result, we show how one can construct a natural RDS based on a given SDS. 
This construction also yields several consequences that are essential for our subsequent analysis.~The next statement is a summary of the results in~\cite{BNGH26}.

\begin{theorem}\label{RDS_SDS}
	Let $\left(\mathcal{B}, \{P_t\}_{t \geq 0}, \mathbf{P}, \{\vartheta_t\}_{t \geq 0}\right)$ be a stationary noise process and let \( \varphi \) be a continuous SDS in the sense of Definition~\ref{SDSSS}. Then the following statements hold. 
	\begin{enumerate}
		\item  Consider the probability space \((\Omega, \mathcal{F}, \mathbb{P})\), where 
		\[
		\Omega := \mathcal{B} \times \mathcal{B}, \quad 
		\mathcal{F} := \sigma(\mathcal{B}) \otimes \sigma(\mathcal{B}), \quad 
		\mathbb{P} := \mathbf{P} \times \mathbf{P}.
		\]
		Define \( \theta_t : \Omega \to \Omega \) by
		\begin{align}\label{semigr}
			\begin{split}
				\theta_t(\omega^-,\omega^+) :=
				\begin{cases}
					\left(P_{t}(\omega^-,\omega^+), \vartheta_{t}\omega^+\right), & \text{for } t \ge 0,\\[4pt]
					\left(\vartheta_{-t}\omega^-, P_{-t}(\omega^+, \omega^-)\right), & \text{for } t < 0.
				\end{cases}
			\end{split}
		\end{align}
		Then \((\Omega, \mathcal{F}, \mathbb{P}, \theta)\) constitutes a measurable metric dynamical system.
		\item Define
		\begin{align}
			\begin{split}
				&\Phi: \mathbb{R}^+ \times \mathcal{X} \times \Omega \to \mathcal{X},\\
				&(t, x, (\omega^-, \omega^+)) \longmapsto \Phi^{t}_{(\omega^-, \omega^+)}(x) := \varphi^{t}_{P_{t}(\omega^-, \omega^+)}(x).
			\end{split}
		\end{align}
		Then \( \Phi \) is a continuous RDS over \((\Omega, \mathcal{F}, \mathbb{P}, \theta)\). If \( \varphi \) is of class \( C^k \) for any \( k \ge 1 \), then the same regularity holds for \( \Phi \).
		
		\item Suppose that \( \varphi \) admits an invariant measure \( \mu \) in the sense of Definition~\ref{DEAFRS}. For every \( t \ge 0 \), define
		\begin{align*}
			&\Theta_{t}: \mathcal{X} \times \Omega \to \mathcal{X} \times \Omega ,\\
			&(x, \omega^-, \omega^+) \longmapsto \left(\Phi^{t}_{(\omega^-, \omega^+)}(x), P_{t}(\omega^-, \omega^+), \vartheta_{t}\omega^+\right).
		\end{align*}
		Then
		\begin{align}\label{MERTCIDY}
			\left(\mathcal{X} \times\Omega, \; 
			\sigma(\mathcal{X}) \otimes\mathcal{F}, \; 
			\mu \times \mathbf{P}, \; \Theta \right)
		\end{align}
		is a metric dynamical system.
		
		\item Let \(\mu\) be an ergodic invariant measure and suppose that for every \(t > 0\), the map \(\vartheta_t : \mathcal{B} \to \mathcal{B}\) is ergodic with respect to the measure \(\mathbf{P}\). 
		Then for every \(t > 0\), the map \(\Theta_t\) is ergodic with respect to \(\mu \times \mathbf{P}\).
		
		\item Suppose that \( \mathcal{X} \) is a separable Banach space and \( \varphi \) is a \( C^k \)-SDS for some \( k \ge 1 \). Define 
		\begin{align*}
			& \mathcal{T} : \mathbb{R}^{+} \times \mathcal{X} \times 
			\left(\mathcal{X} \times \Omega \right) \to \mathcal{X},\\
			&\left(t, y, (x, \omega^-, \omega^+)\right) \longmapsto 
			\mathcal{T}^{t}_{(x,\omega^-,\omega^+)}(y) := \mathrm{D}_{x}\Phi^{t}_{(\omega^-,\omega^+)}(y).
		\end{align*}
		Then \( \mathcal{T} \) is a linear continuous RDS over the metric dynamical system~\eqref{MERTCIDY}.
	\end{enumerate}
\end{theorem}
\proof
The proofs are stated in \cite{BNGH26}. The first two statements follow from Theorem 3.2 and Remark 3.3,  the third one from Corollary 3.5, and the last two claims follow from Corollaries 3.6 and 3.7.
\qed
The next result shows that an invariant measure of an SDS (which exists on $\cX\times \cB$) can be disintegrated, yielding a family of random measures on the state space \(\mathcal{X}\).
\begin{lemma}\label{RAND_MEASURE}
	Suppose that $\varphi$ is an SDS with an invariant measure $\mu$.
	Then, there exists a set $\tilde{\mathcal{B}} \subseteq \mathcal{B}$ of full measure and a family of random measures $\{{\mu}_{\omega^-}\}_{\omega^- \in \mathcal{B}}$ such that, for every $A \times C \in \sigma(\mathcal{X}) \otimes \sigma(\mathcal{B})$, the following properties hold:
	\begin{enumerate}
		\item ${\mu}_{\omega^-}(\mathcal{X}) = 1$ for $\mathbf{P}$-almost every $\omega^- \in \mathcal{B}$.
		\item The map $\omega^- \mapsto {\mu}_{\omega^-}(A)$ is measurable.
		\item The measure $\mu$ admits the disintegration
		\[
		\mu(A \times C) = \int_{C} {\mu}_{\omega^-}(A)\, \mathbf{P}(\mathrm{d}\omega^-).
		\]
		\item For all $(\omega^-, \omega^+) \in \tilde{\mathcal{B}} \times \mathcal{B}$,
		\begin{align}\label{eq:invariance}
			\begin{split}
				&P_t(\omega^-, \omega^+) \in \tilde{\mathcal{B}},\\
				&\left(\varphi_{P_t(\omega^-, \omega^+)}^{t}\right)_{\ast} {\mu}_{\omega^-}
				= {\mu}_{P_t(\omega^-, \omega^+)}.
			\end{split}
		\end{align}
		\item \label{PRPR} 
		For every $t \geq 0$,
		\begin{align}
			\mathbb{P}(\tilde{\mathcal{B}} \times \mathcal{B}) = 1, 
			\qquad  
			\theta_{t}(\tilde{\mathcal{B}} \times \mathcal{B}) \subseteq \tilde{\mathcal{B}} \times \mathcal{B}.
		\end{align}
		Moreover, for every $(\omega^-, \omega^+) \in \tilde{\mathcal{B}} \times \mathcal{B}$ and every $t \geq 0$,
		\begin{align*}
			(\Phi^{t}_{(\omega^-, \omega^+)})_{\star} \mu_{\omega^-} 
			= \mu_{P_{t}(\omega^-, \omega^+)}.
		\end{align*}
	\end{enumerate}
\end{lemma}
\proof 
We briefly sketch the main idea of the proof and refer to~\cite[Lemma~3.8 and Theorem~3.11]{BNGH26} for more details. First, we use the fact that the measure \(\mu\) can be disintegrated into a family of probability measures \(\{\tilde{\mu}_{\omega^-}\}_{\omega^- \in \mathcal{B}}\) on \(\mathcal{X}\). Then, we apply an averaging argument to modify this family of measures on a set $\tilde{\cB}\subset \cB$ of full measure on which the properties stated above hold. This argument is closely related to the perfection technique in the theory of random dynamical systems.
\qed
In this manuscript, we work with RDS generated by SDEs driven by fractional Brownian motion making their long-time behavior a challenging task.~We provide a short summary of some well-known concepts for Markovian dynamics for a better comprehension of our results. 

\begin{remark}\label{RESFSS}
	For SDEs driven by Brownian motion, the corresponding RDS denoted by $\overline{\varphi}$, 
	where $\overline{\varphi} \colon \mathbb{R}_+ \times \overline{\Omega} \times \R^d \to \R^d$ 
	and $\R^d$ is a separable Banach space, 
	is defined on a measurable metric dynamical system 
	$(\overline{\Omega}, \overline{\mathcal{F}}, \overline{\mathbb{P}}, \overline{\theta})$ 
	associated with the law of the Brownian motion.
	This RDS is a white-noise RDS (see \cite[Section~1.1]{CGS16} for the definition). 
	It is adapted to a suitable filtration and satisfies the usual independence and shift-invariance properties of the Brownian motion,
	see \cite[Section~1.1]{CGS16} for further details. 
	The RDS $\overline{\varphi}$ induces a Markovian (Feller) semigroup $(\overline{P}_t)_{t\geq 0}$ defined by
	\[
	\overline{P}_t f(x) := \overline{\mathbb{E}}\big[f\big(\overline{\varphi}(t, \cdot, x)\big)\big],
	\]
	for every bounded and measurable function $f:\R^d\to \R$. 
	If $(\overline{P}_t)_{t\geq 0}$ admits an invariant probability measure $\overline{\mu}$, 
	by the martingale convergence theorem 
	(see e.g.~\cite[Corollary~4.6]{CF94} and \cite[Chapter~5]{V14}), 
	the corresponding statistical equilibrium $\overline{\mu}_{\overline{\omega}}$ exists. 
	More precisely, for every sequence $t_k \to \infty$, the weak$^\ast$ limit
	\[
	\overline{\mu}_{\overline{\omega}} := 
	\lim_{t_k \to \infty} 
	\big(\overline{\varphi}(t_k, \overline{\theta}_{-t_k}\overline{\omega}, \cdot)\big)_{\ast} \overline{\mu}
	\]
	exists $\overline{\mathbb{P}}$-almost surely and does not depend on the particular choice of the sequence $(t_k)$. 
	Moreover, for every $t \ge 0$,
	\[
	\big(\overline{\varphi}(t, \overline{\omega}, \cdot)\big)_{\ast} \overline{\mu}_{\overline{\omega}} 
	= \overline{\mu}_{\overline{\theta}_t \overline{\omega}},
	\qquad 
	\overline{\mathbb{E}}\,\overline{\mu}_{\overline{\omega}} = \overline{\mu}.
	\]
	The correspondence between invariant measure of a Markov process and statistical equilibrium for the corresponding white-noise RDS does not hold in the non-Markovian case considered here. Therefore we follow the approach developed in Lemma \ref{RAND_MEASURE}. 
	
\end{remark}
\begin{remark}\label{ARATSasd}
	A Markovian semigroup $(\overline{P}_t)_{t \geq 0}$ with invariant measure $\overline{\mu}$ is called {strongly mixing} if
	\[
	\overline{P}_t f(x) \longrightarrow \int_{\cX} f(y)\, \overline{\mu}(dy)
	\quad \text{as } t \to \infty,
	\]
	for every bounded and continuous function $f$ and all $x \in \cX$.
	Similarly, an RDS $\overline{\varphi}$ is called strongly mixing if the law of $\overline{\varphi}(t, \cdot, x)$
	converges to $\overline{\mu}$ for $t \to \infty$ for all $x \in \cX$. For a strongly mixing, white-noise RDS, it is well-known that the family $\{ \overline{\mu}_{\overline{\omega}} \}_{\overline{\omega} \in \overline{\Omega}}$
	constitutes a weak point attractor.
	For more details, see \cite[Proposition~2.20]{FGS16a} and \cite[Theorem~2.4]{KS04}.
\end{remark}
\subsection{Fractional Brownian motion}
The primary motivation of the theory of stochastic dynamical systems is to analyze the long-time behavior of stochastic differential equations driven by fractional Brownian motion. Due to the Mandelbrot-van Ness representation, such equations naturally fall within this framework. We now briefly recall some  results mainly following \cite{Hai05} and \cite{BNGH26}.
\begin{definition}\label{BB_HH}
	For \( 0 < H < 1 \), let \( C_0^\infty(\mathbb{R}^-, \mathbb{R}^d) \) denote the space of smooth, compactly supported functions \( \omega: \mathbb{R}^- \to \mathbb{R}^d \) satisfying \( \omega(0) = 0 \). For \( \omega \in C_0^\infty(\mathbb{R}^-, \mathbb{R}^d) \), define
	\begin{align}\label{norm:BH}
		\Vert \omega \Vert_{\mathcal{B}_H}
		:= \sup_{s,t \in \mathbb{R}^-} 
		\frac{|\omega(t) - \omega(s)|}
		{\sqrt{1 + |t| + |s|}\, |t-s|^{\frac{1-H}{2}}}.
	\end{align}
	The space \( \mathcal{B}_H \) is defined as the closure of \( C_0^\infty(\mathbb{R}^-, \mathbb{R}^d) \) with respect to this norm and is a separable Banach space.
\end{definition}

We can now state the following result.

\begin{lemma}\label{AJSJAsd}
	Let \( H \in (0,1) \). For \( \omega \in C_0^\infty(\mathbb{R}^-, \mathbb{R}^d) \), define
	\begin{align}\label{AAS11}
		\mathcal{D}_H \omega(t)
		:= \frac{1}{\alpha_H} \int_{-\infty}^{0}
		(-r)^{H - \frac{1}{2}}
		\left( \dot{\omega}(t + r) - \dot{\omega}(r) \right)
		\, \mathrm{d}r, 
		\quad t \le 0.
	\end{align}
	This operator extends continuously from \(C_0^\infty(\mathbb{R}^-, \mathbb{R}^d) \) to \( \mathcal{B}_H \) and satisfies the following properties.
	\begin{itemize}
		\item The map
		\begin{align}\label{D_HHF}
			\mathcal{D}_H : \mathcal{B}_H \longrightarrow \mathcal{B}_{1-H}
		\end{align}
		admits a bounded inverse.
		\item There exists a unique Gaussian measure \( \mathbf{P} \) on \( \mathcal{B}_H \) such that the associated canonical process is a time-reversed Brownian motion.
		\item The canonical process associated with the pushforward measure \( (\mathcal{D}_H)_{\ast} \mathbf{P} \) on \( \mathcal{B}_{1-H} \) is a time-reversed fractional Brownian motion with Hurst parameter \( H \).
		\item The quadruple \( \left( \mathcal{B}_H, \{P_t\}_{t \ge 0}, \mathbf{P}, \{\vartheta_t\}_{t \ge 0} \right) \) forms a stationary noise process in the sense of Definition~\ref{SDS}, where:
		\begin{enumerate}
			\item the shift maps \( \{\vartheta_t\}_{t \ge 0} \) on \( \mathcal{B}_H \) are defined by
			\begin{align}\label{PPOOAS1}
				\vartheta_t \omega(s) := \omega(s - t) - \omega(-t), \quad s \le 0.
			\end{align}
			\item for every \( t \ge 0 \), the map \( P_t : \mathcal{B}_H \times \mathcal{B}_H \to \mathcal{B}_H \) is given by
			\begin{align}\label{PPOOAS}
				P_t(\omega^-, \omega^+)(s)
				:= 
				\begin{cases}
					\omega^+(-s - t) - \omega^+(-t), & -t \le s \le 0, \\[4pt]
					\omega^-(s + t) - \omega^+(-t), & s \le -t.
				\end{cases}
			\end{align}
		\end{enumerate}
	\end{itemize}
\end{lemma}
\proof This follows from~\cite[Lemmas~3.6, 3.8, and~3.10]{Hai05} and~\cite[Lemma~2.18]{BNGH26}.
\qed
We also have the following result.

\begin{lemma}\label{RR_TT}
	Consider the setting of Lemma~\ref{AJSJAsd}. Then the following statements are valid.
	\begin{itemize}
		\item For every \( \omega^-, \omega^+ \in \mathcal{B}_H \) and all \( 0 \leq \tau \leq s \leq t \), the following identity holds:
		\begin{align}\label{RR_TT1}
			\begin{split}            
				\mathcal{D}_{H} P_{t}(\omega^-, \omega^+)(\tau - t)
				&- \mathcal{D}_{H} P_{t}(\omega^-, \omega^+)(-t)
				\\&= \mathcal{D}_{H} P_{s}(\omega^-, \omega^+)(\tau - s)
				- \mathcal{D}_{H} P_{s}(\omega^-, \omega^+)(-s).
			\end{split}
		\end{align}
		
		\item For all \( r \leq 0 \) and \( \omega^- \in \mathcal{B}_H \), we have 
		\begin{align}\label{RR_TT2}
			(\mathcal{D}_{H} \vartheta_{t} \omega^-)(r)
			= \mathcal{D}_{H} \omega^-(-t + r)
			- \mathcal{D}_{H} \omega^-(-t).
		\end{align}
		
		Define the process \( (B_t^H)_{t \in \mathbb{R}} \) by
		\begin{align*}
			B_t^H =B_{t}^{H}(\omega^-,\omega^+):=
			\begin{cases}
				- \bigl( \mathcal{D}_H P_t(\omega^-, \omega^+) \bigr)(-t), & t \geq 0, \\[6pt]
				\mathcal{D}_H \omega^-(t), & t \leq 0.
			\end{cases}
		\end{align*}
		Then \( (B_t^H)_{t \in \mathbb{R}} \) denotes a two-sided fractional Brownian motion with Hurst parameter \( H \), adapted to the canonical filtration \( (\mathcal{F}_{-\infty}^t)_{t \in \mathbb{R}} \) generated by the underlying Brownian motion appearing in the Mandelbrot-van Ness representation.
	\end{itemize}
\end{lemma}
\proof
We refer to~\cite[Lemmas~2.19~and~2.20]{BNGH26}.
\qed

\subsection{SDEs driven by fractional Brownian motion}
Finally, we briefly explain how the theory of stochastic dynamical systems can be applied to the SDE 
\begin{align}\label{MAIN}
	\begin{cases}
		\mathrm{d}Y_t = F(Y_t)\, \mathrm{d}t + \sigma\, \mathrm{d}B^H_t, \\
		Y_0 = x \in \mathbb{R}^d.
	\end{cases}
\end{align}
\begin{assumptions}\label{Drift}
	We impose the following conditions on the coefficients.
	\begin{itemize}
		
		\item [1)] The mapping \( \sigma : \mathbb{R}^d \to \mathbb{R}^d \) is an invertible matrix.
		
		\item [2)] \label{DISISV} For all \( \xi_1, \xi_2 \in \mathbb{R}^d \),
		\[
		\langle F(\xi_2) - F(\xi_1),\, \xi_2 - \xi_1 \rangle 
		\leq 
		\min \{ C_1^F - C_2^F |\xi_2 - \xi_1|^2,\; C_3^F |\xi_2 - \xi_1|^2 \},
		\]
		where \( C_i^F > 0 \) for \( i \in \{1,2,3\} \).
		
		\item [3)] The function \( F \) is differentiable and there exist constants \( C_F > 0 \) and \( N \geq 1 \) such that
		\[
		|F(\xi)| \leq C_F(1 + |\xi|)^N, 
		\qquad 
		|\mathrm{D}_\xi F| \leq C_F(1 + |\xi|)^N,
		\quad \text{for all } \xi \in \mathbb{R}^d.
		\]
		\item [4)] If \( \frac{1}{2} < H < 1 \), the derivative of \( F \) is globally bounded.
	\end{itemize}
\end{assumptions}
The next result states how to obtain an RDS (or SDS) from the solution of the SDE \eqref{MAIN}.
\begin{theorem}\label{RDS_SDS_MEAU}
	Under the assumptions of Lemma~\ref{RR_TT}, for every $\omega = (\omega^-, \omega^+) \in \mathcal{B}_H \times \mathcal{B}_H$ and $x \in \mathbb{R}^d$, 
	let us denote the solution of the SDE~\eqref{MAIN} by $(\Phi^{t}_{(\omega^-, \omega^+)}(x))_{t \ge 0}$, i.e.,
	\begin{align*}
		\Phi^{t}_{(\omega^-, \omega^+)}(x) = x + \int_{0}^{t} F\big(\Phi^{s}_{(\omega^-, \omega^+)}(x)\big)\, \mathrm{d}s + \sigma B_{t}^{H}(\omega).
	\end{align*}
	Then the following assertions hold for $t \geq 0$, $x \in \mathbb{R}^d$ and $\omega = (\omega^-, \omega^+) \in \mathcal{B}_H \times \mathcal{B}_H$.
	\begin{enumerate}
		\item  
		The value $\Phi^{t}_{\theta_{-t}(\omega^-, \omega^+)}(x)$ is independent of $\omega^+$ and depends only on $\omega^-$. In particular, setting
		\[
		\varphi^{t}_{\omega^-}(x) := \Phi^{t}_{\theta_{-t}(\omega^-, \omega^+)}(x),
		\]
		where the shift map $(\theta_t)_{t \geq 0}$ is defined in~\eqref{semigr}, entails a \( C^1 \)-SDS on \( \mathcal{X} = \mathbb{R}^d \) with the stationary noise process introduced in Lemma~\ref{AJSJAsd}. 
		
		\item It holds that $\Phi^{t}_{(\omega^-, \omega^+)}(x) = \varphi^{t}_{P_{t}(\omega^-, \omega^+)}(x)$. 
		Moreover, $\Phi$ is a $C^1$-RDS by Theorem~\ref{RDS_SDS}.  
		\item The SDS~$\varphi$ admits a unique invariant (and hence ergodic) measure~$\mu$ in the sense of Definition~\ref{DEAFRS}.  
		In addition, for the law of the solution $(\Phi^{t}_{(\omega^-, \omega^+)}(x))_{t \geq 0}$, we have the convergence
		\begin{align}\label{TOTAL}
			\mathcal{L}\left(\Phi^{t}_{(\omega^-, \omega^+)}(x)\right) \xrightarrow{t \to \infty} (\Pi_{\mathbb{R}^d})_{\star} \mu,
		\end{align}
		in total variation.
	\end{enumerate}
\end{theorem}
\proof
This was established in~\cite[Lemma~4.4, Corollary~4.8, and Proposition~4.11]{BNGH26}.
\qed
\section{Random invariant measures and attractors for RDS}\label{sec:3}
In this section, we aim to describe the structure of the random measures that constitute the random fibers of the invariant measure \( \mu_{\omega^-} \). Moreover, we prove the existence of random attractors which follows from the dissipativity of \( F \) stated in Assumption~\ref{Drift}. Moreover, we explore the connection between invariant measures and attractors; see also Remarks~\ref{RESFSS} and \ref{ARATSasd}.
We fix here some useful notations that will be used further on.\begin{notation}\label{ASSMM}
	\begin{itemize}
		\item We fix $H \in (0,1)$, write \( \mathcal{B} \) instead of \( \mathcal{B}_H \), and work on the probability space $(\Omega,\cF,\mathbb{P})$, where \( \Omega = \mathcal{B} \times \mathcal{B} \) and \( \mathbb{P} = \mathbf{P} \times \mathbf{P} \) as introduced in Theorem \ref{RDS_SDS}. 
		
		
		
		\item By \( \tilde{\mathcal{B}} \) we refer to the measurable subset of \( \mathcal{B} \) obtained in  Lemma~\ref{RAND_MEASURE}.
	\end{itemize}
\end{notation}  

\subsection{Random invariant measures}
From Lemma \ref{RAND_MEASURE}, one can disintegrate the invariant measure of the SDS to obtain a family of random measures on the phase space $\mathbb{R}^d$. The aim of this subsection is to investigate the structure of these random measures, when the top Lyapunov exponent is negative.
We recall some preliminary results in this direction following \cite{BNGH26}, in particular the existence of a local stable manifold. 
\begin{theorem}\label{ASASiid}
	There exists a sequence of deterministic values, called {Lyapunov exponents}
	\[
	\lambda_k < \cdots < \lambda_1, \qquad \lambda_i \in [-\infty, \infty),
	\]
	and a set of full measure $\tilde{\Omega} \subseteq \mathbb{R}^d \times \mathcal{B} \times \mathcal{B}$ with respect to $\mu \times \mathbf{P}$ such that 
	\begin{enumerate}
		\item The set $\tilde{\Omega}$ is $\Theta_t$-invariant for every $t \ge 0$.  
		For each $\lambda \in \mathbb{R}$, define
		\[
		G_\lambda(\omega,x) := \left\{ z \in \mathbb{R}^d : 
		\limsup_{t \to \infty} \frac{1}{t} \log \big| \mathrm{D}_x \Phi^t_\omega(z) \big| \le \lambda \right\}.
		\]
		Then $G_{\lambda_1}(\omega, x) = \mathbb{R}^d$, and for every $1 \le i < k$,
		\[
		z \in G_{\lambda_i}(\omega,x) \setminus G_{\lambda_{i+1}}(\omega,x)
		\quad \Longleftrightarrow \quad
		\lim_{t \to \infty} \frac{1}{t} \log \big| \mathrm{D}_x \Phi^t_\omega(z) \big| = \lambda_i.
		\]
		Moreover, for all $z \in G_{\lambda_k}(\omega,x) \setminus \{0\}$,
		\[
		\lim_{t \to \infty} \frac{1}{t} \log \big| \mathrm{D}_x \Phi^t_\omega(z) \big| = \lambda_k.
		\]
		\item Suppose $\lambda_1 < 0$ and that for some $r \in (0,1]$ and all $y,z \in \mathbb{R}^d$ there exist constants $\overline{C}_F > 0$ and $p_1 \ge 1$ such that
		\begin{equation}\label{DEEE}
			\| \mathrm{D}_y F - \mathrm{D}_z F \|_{L(\mathbb{R}^d, \mathbb{R}^d)}
			\le \overline{C}_F (1 + |y| + |z|)^{p_1} |y-z|^r.
		\end{equation}
		Then there exists a full-measure set $\mathcal{M} \subseteq \mathbb{R}^d \times \Omega$ such that, for every $0 < \nu < -\lambda_1$, there exists a positive random variable
		\begin{align}\label{RRRNU}
			R^\nu : \mathcal{M} \to (0,\infty)
		\end{align}
		with the property that for all $(x,\omega) \in \mathcal{M}$ and all $y \in \mathbb{R}^d$ satisfying $|y| \le R^\nu(x,\omega)$,
		\[
		\sup_{t \ge 0} \exp(t\nu) \, \big| \Phi^t_\omega(x+y) - \Phi^t_\omega(x) \big| < \infty.
		\]
	\end{enumerate}
\end{theorem}
\proof
See \cite[Propositions~4.13 and~4.17]{BNGH26}.
\qed
This leads to the following consequence.
\begin{corollary}\label{STABSS}
	Let the hypotheses of the second statement of Theorem \ref{ASASiid} hold. Then for every \( \epsilon > 0 \), we can find a random, non-empty, open set \( U(\omega) \subset \mathbb{R}^d \), such that for
	\begin{align}\label{shrink}
		\mathcal{C}:=\left\lbrace \omega : \ \ \lim_{n \to \infty} \text{\emph{diam}}\left(\Phi^{n}_{\omega}(U(\omega))\right) = 0 \right\rbrace,
	\end{align}
	we have \( {\P}(\mathcal{C}) > 1 - \epsilon \), where the diameter of a subset $S\subseteq \R^d$ is defined as \( \text{\emph{diam}}(S)=\sup\{ |x-y| : x,y\in S\} \). 
\end{corollary}
\proof
Let \( R^{\nu} \) be the random variable obtained in Theorem \ref{ASASiid}. For each \( n \in \mathbb{N} \), we define
\[
\overline{\mathcal{M}}^{n} := \left\lbrace (x, \omega) \in \mathcal{M} : R^{\nu}(x, \omega) \geq \frac{1}{n} \right\rbrace.
\]
By construction, the set \( \overline{\mathcal{M}}^{n} \) is measurable and satisfies
\[
\lim_{n \to \infty} (\mu \times \mathbf{P})(\overline{\mathcal{M}}^n) = 1.
\]
Since \( \mathbb{R}^d \times \Omega \) is a Polish space, every probability measure on it is regular. Hence, for each \( n \in \mathbb{N} \), there exists a compact subset \( \mathcal{M}^n \subseteq \overline{\mathcal{M}}^n \) such that
\[
\lim_{n \to \infty} (\mu \times \mathbf{P})(\mathcal{M}^n) = 1.
\]
In particular, the projection \( \Pi_\Omega(\mathcal{M}^n) \) is compact and thus a measurable subset of \( \Omega \), so we obtain
\begin{align}\label{YHa}
	\lim_{n \to \infty} \P\left( \Pi_\Omega(\mathcal{M}^n) \right) = 1.
\end{align}
Furthermore, for every \( \omega \in \Pi_\Omega(\mathcal{M}^n) \), the fiber 
\begin{align}\label{GBAs85}
	{\mathcal{M}^n}(\omega) = \left\lbrace x \in \mathbb{R}^d : (x, \omega) \in {\mathcal{M}^n} \right\rbrace
\end{align}
is a compact and measurable subset of \( \mathbb{R}^d \). By Lemma~\ref{MESUA}, there exists a Borel-measurable function \( \mathcal{K} : \Pi_\Omega(\mathcal{M}^n) \to \mathbb{R}^d \) such that \( \mathcal{K}(\omega) \in \mathcal{M}^n(\omega) \) for every \( \omega \in \Pi_\Omega(\mathcal{M}^n) \).
Moreover, by definition
\[
R^\nu(\mathcal{K}(\omega), \omega) \geq \frac{1}{n}, \quad \text{for all } \omega \in \Pi_\Omega(\mathcal{M}^n).
\]
We further define the open ball in $\R^d$
\[
U(\omega) := \left\lbrace x \in \mathbb{R}^d : \left\vert x - \mathcal{K}(\omega) \right\vert < R^\nu(\mathcal{K}(\omega), \omega) \right\rbrace.
\]
Due to \eqref{YHa} 
and  Theorem~\ref{ASASiid}, we conclude that the  claim holds for the family \( \{U(\omega)\}_{\omega \in \Pi_\Omega(\mathcal{M}^n)} \), provided that \( n \) is chosen sufficiently large.
\qed\\
Recalling Remark \ref{RESFSS}, we know that we can construct an invariant measure for a white-noise RDS in the Markovian setting. In that case, the sign of the first Lyapunov exponent determines the nature of the associated random measures: negative exponents yield atomic measures, while positive ones lead to SRB measures. 
Since we assume a negative top Lyapunov exponent, it is natural to expect that the corresponding measure is atomic which is what we prove next. The foundation of our argument originates from~\cite{G87}, albeit with several important modifications. We also refer to~\cite[Lemma 2.19]{FGS16a}, where a related problem is studied in the context of Markov processes based on the techniques from~\cite{G87}.
\begin{proposition}\label{descrte} 
	Let the hypotheses of the second statement in Theorem \ref{ASASiid} hold. Then there exist \( p \in \mathbb{N} \) and random variables \( \{ b_i \}_{1 \leq i \leq p} \), with \( b_i : \tilde{\mathcal{B}} \rightarrow \mathbb{R}^d \), such that
	\[
	\mu_{\omega^-} = \frac{1}{p} \sum_{i=1}^{p} \delta_{\{ b_i(\omega^-) \}},
	\]
	where \( \mu_{\omega^-} \) denotes the disintegrated invariant measure obtained in Lemma~\ref{RAND_MEASURE}. In particular, this representation shows that the disintegrated measures are discrete.
\end{proposition}
\proof
For $(x,\omega)\in\R^d\times\Omega$ , we define
\begin{equation}\label{SSSdsd}
	\mathcal{H}(x, \omega^-, \omega^+) := \mu_{\omega^-}(\{x\}).
\end{equation}
Note that since \(\Phi^{t}_{\omega}\) is injective, we have $  (\Phi^{t}_{\omega})^{-1}\big(\{\Phi^{t}_{\omega}(x)\}\big) = \{x\}$. Thus, by property \ref{PRPR} 
of Lemma \ref{RAND_MEASURE},
\begin{align*}
	\mathcal{H}\left(\Theta_{t}(x,\omega)\right)&=\mu_{P_{t}(\omega^-,\omega^+)}\left(\{\Phi^{t}_{\omega}(x)\}\right)=\left(\Phi^{t}_{\omega}\right)_{\star}\mu_{\omega^-}\left(\{\Phi^{t}_{\omega}(x)\}\right)\\&=\mu_{\omega^-}(\Phi^{t}_{\omega})^{-1}(\{\Phi^{t}_{\omega}(x)\})=\mathcal{H}(x,\omega).
\end{align*} 
Since for every $t>0$, \( \Theta_t\) is ergodic, the function \( \mathcal{H} \) must be constant on a set of full measure with respect to \( \mu \times \mathbf{P} \). There are two possible scenarios: \( \mathcal{H} \neq 0 \) or \( \mathcal{H} = 0 \). We treat these two cases separately. First, we show that assuming \( \mathcal{H} \neq 0 \) proves our claim and second we demonstrate that the case \( \mathcal{H} = 0 \) can be ruled out. \\
\textbf{Case I}. This is similar to the argument provided in \cite[Lemma 2.19]{FGS16a}. However, since in our setting \( \Phi \) is not Markov, we provide more details for the sake of completeness. We prove the statement by contradiction and 
assume that \( \mathcal{H} \neq 0 \). Then, almost surely for \( \omega^- \in \mathcal{B} \), the support of \( \mu_{\omega^-} \) must consist of finitely many discrete points. To see this, we define
\[
N(\omega^-) :=
\begin{cases}
	\# \operatorname{supp}(\mu_{\omega^-}), & \text{if } \# \operatorname{supp}(\mu_{\omega^-}) < \infty \\[6pt]
	+\infty, & \text{otherwise}.
\end{cases}
\]
We have
\begin{align*}
	\mathcal{H} = \int_{\mathbb{R}^d \times \Omega} \mathcal{H}(x,\omega) \, (\mu \times \mathbf{P})(\mathrm{d}(x,\omega)) = \int_{\mathcal{B}} \int_{\mathbb{R}^d} \mu_{\omega^-}(\{x\}) \, \mu_{\omega^-}(\mathrm{d}x) \, \mathbf{P}(\mathrm{d}\omega^-).
\end{align*}
Since \( \mathcal{H}(x,\omega) \) is constant \( \mu \times \mathbf{P} \)-almost surely, we have
\begin{align*}
	\mathcal{H} = \int_{\mathcal{B}} \int_{\mathbb{R}^d} \mu_{\omega^-}(\{x\}) \, \mu_{\omega^-}(\mathrm{d}x) \, \mathbf{P}(\mathrm{d}\omega^-) = \mathcal{H}^2 \int_{\mathcal{B}} N(\omega^-) \, \mathbf{P}(\mathrm{d}\omega^-).
\end{align*}
We know that \(\mathcal{H} \neq 0\). Thus
\[
\int_{\mathcal{B}} N(\omega^-) \, \mathbf{P}(\mathrm{d}\omega^-) = \frac{1}{\mathcal{H}} < \infty.
\]
Moreover by Fubini's theorem, the function \( \omega^- \mapsto N(\omega^-) \) is measurable. Since \( \mathcal{H}(x,\omega) \) is constant with respect to \( \mu \times \mathbf{P} \), this indicates that \( N(\omega^-) \) is \( \mathbf{P} \)-almost surely equal to some \( p \in \mathbb{N} \) and so $p=\frac{1}{\mathcal{H}}$.
Consequently, on a set of full measure in \( \mathcal{B} \), we have
\[
\mu_{\omega^-} = \frac{1}{p} \sum_{i=1}^{p} \delta_{\{ b_{i}(\omega^-) \}},
\]
where \( (b_{i}(\omega^-))_{1 \leq i \leq p} \) are points in the support of \( \mu_{\omega^-} \).

\textbf{Case II.} Now, we consider the case $\cH= 0$.   Then, almost surely, all random measures contain no point masses and are therefore diffuse. More precisely, there exists a set of full measure in \( \tilde{\mathcal{B}} \) such that for every \( x \in \mathbb{R}^d \), we have
\begin{align*}
	\mu_{\omega^-}\left(\{x\}\right) = 0.
\end{align*}
We prove that this leads to a contradiction. Consequently, the assertion follows from the previous case.
Let 
\begin{align*}
	\Delta: = \left\lbrace (x,x) \mid x \in \mathbb{R}^d \right\rbrace
\end{align*}
be the diagonal in $\R^d \times \R^d$.
Then we can find a measurable function 
\[
G: \mathbb{R}^d \times \mathbb{R}^d \setminus \Delta \to [0, \infty)
\]
such that  
\begin{align*}
	\begin{split}
		&\int_{\mathcal{B}}\int_{(\mathbb{R}^d \times \mathbb{R}^d) \setminus \Delta} G(x,y) \, \mu_{\omega^-}(\mathrm{d}x) \mu_{\omega^-}(\mathrm{d}y)\ \mathbf{P}(\mathrm{d}\omega^-) \\&=\int_{\mathcal{B}}\int_{\tilde{\mathcal{B}}}\int_{(\mathbb{R}^d \times \mathbb{R}^d) \setminus \Delta} G(x,y) \, \mu_{\omega^-}(\mathrm{d}x) \mu_{\omega^-}(\mathrm{d}y)\ \mathbf{P}(\mathrm{d}\omega^-)\mathbf{P}(\mathrm{d}\omega^+)< \infty
	\end{split}
\end{align*}
and  
\begin{align}\label{GBAss}
	G(x,y) \to \infty \quad \text{as} \quad \vert x - y \vert \to 0.
\end{align}
For the construction of such a function we refer to \cite[Lemma 2.19]{FGS16a}. 
We further set
\begin{align*}
	R(\omega^-,\omega^+):=\int_{(\mathbb{R}^d \times \mathbb{R}^d) \setminus \Delta} G(x,y) \, \mu_{\omega^-}(\mathrm{d}x) \mu_{\omega^-}(\mathrm{d}y).
\end{align*}
Then, from Lemma \ref{pushforward} and the fact that $(\theta_t)_{t\geq 0}$ is $\mathbf{P}\times\mathbf{P}$ invariant, 
it follows for every \( t > 0 \) that
\begin{align}\label{TGB2}
	\begin{split}
		&\int_{\mathcal{B}}\int_{\tilde{\mathcal{B}}}\int_{(\mathbb{R}^d \times \mathbb{R}^d) \setminus \Delta} G(x,y) \, \mu_{\omega^-}(\mathrm{d}x) \mu_{\omega^-}(\mathrm{d}y)\ \mathbf{P}(\mathrm{d}\omega^-)\mathbf{P}(\mathrm{d}\omega^+)\\&=\int_{\mathcal{B}}\int_{\tilde{\mathcal{B}}}R(\omega^-,\omega^+)\mathbf{P}\times \mathbf{P}\left(\mathrm{d}(\omega^-,\omega^+)\right)=\int_{\mathcal{B}}\int_{\tilde{\mathcal{B}}}R\left(\theta_{t}(\omega^-,\omega^+)\right)\mathbf{P}\times \mathbf{P}\left(\mathrm{d}(\omega^-,\omega^+)\right).
	\end{split}
\end{align}
From Lemma \ref{RAND_MEASURE} we have $(\Phi^{t}_{\omega})_{\star}\mu_{\omega^-} \;=\; \mu_{P_{t}(\omega^-,\omega^+)}.
$
With a slight abuse of notation, this yields
\[
(\Phi^{t}_{\omega})_{\star} \left(\mu_{\omega^-} \times \mu_{\omega^-} \right) = \mu_{P_{t}(\omega^-,\omega^+)} \times \mu_{P_{t}(\omega^-,\omega^+)}.
\]
Therefore
\begin{align*}
	\begin{split}
		&R\left(\theta_{t}(\omega^-,\omega^+)\right)=\int_{(\mathbb{R}^d \times \mathbb{R}^d) \setminus \Delta} G(x,y) \, (\Phi^{t}_{\omega})_{\star}\left(\mu_{\omega^-}(\mathrm{d}x) \mu_{\omega^-}(\mathrm{d}y)\right)\\&=\int_{(\mathbb{R}^d \times \mathbb{R}^d) \setminus \Delta} G\left(\Phi^{t}_{\omega}(x),\Phi^{t}_{\omega}(y)\right)\mu_{\omega^-}(\mathrm{d}x) \mu_{\omega^-}(\mathrm{d}y).
	\end{split}
\end{align*}
Consequently, from 
\eqref{TGB2} we deduce that 
\begin{align}\label{GABSs}
	\begin{split}
		&\int_{\mathcal{B}}\int_{\tilde{\mathcal{B}}}\int_{(\mathbb{R}^d \times \mathbb{R}^d) \setminus \Delta} G(x,y) \, \mu_{\omega^-}(\mathrm{d}x) \mu_{\omega^-}(\mathrm{d}y)\ \mathbf{P}(\mathrm{d}\omega^-)\mathbf{P}(\mathrm{d}\omega^+)\\&=\quad \int_{\mathcal{B}}\int_{\tilde{\mathcal{B}}}\int_{(\mathbb{R}^d \times \mathbb{R}^d) \setminus \Delta} G\left(\Phi^{t}_{\omega}(x),\Phi^{t}_{\omega}(y)\right)\mu_{\omega^-}(\mathrm{d}x) \mu_{\omega^-}(\mathrm{d}y)\ \mathbf{P}(\mathrm{d}\omega^-)\mathbf{P}(\mathrm{d}\omega^+)<\infty.
	\end{split}
\end{align}
In particular, for every $\epsilon > 0$, we can choose $n \in \mathbb{N}$ sufficiently large such that for $\mathcal{M} := \mathcal{M}^n$, with $\mathcal{M}^n$ constructed in the proof of Corollary~\ref{STABSS}, we have 
$(\mu \times \mathbf{P})(\mathcal{M}) > 1 - \epsilon$. Note that we have:
\begin{enumerate}
	\item $R^{\nu}(x, \omega) \geq \frac{1}{n}$ for every $(x, \omega) \in \mathcal{M}$,
	\item $\Pi_{\Omega}(\mathcal{M})$ is measurable.
\end{enumerate}
Thus, by Fubini's theorem, we obtain
\begin{align}\label{8sdde}
	\int_{\Pi_{\Omega}(\mathcal{M})} \mu_{\omega^-}\left(\mathcal{M}(\omega)\right)\ \mathbf{P} \times \mathbf{P}(\mathrm{d}(\omega^-, \omega^+)) > 1 - \epsilon.
\end{align}
Since \( \mu_{\omega^-} \) does not have point masses, it follows once again from Fubini's theorem that
\[
(\mu_{\omega^-} \times \mu_{\omega^-})(\Delta) = 0 \quad \mathbf{P}\text{-a.s.}
\]
Let \( \{D_i\}_{i \geq 1} \) be a countable family of rectangles covering \( \mathbb{R}^d \times \mathbb{R}^d \), each with diameter less than \( \frac{1}{n} \) such that 
\begin{align*}
	D_i \cap D_j \subset \partial D_i \cap \partial D_j, \quad i \neq j.
\end{align*}
Since \( \mu_{\omega^-} \) has no atoms, it follows from Fubini's theorem for every \( i \geq 1 \) that

\begin{align}\label{sddsssw}
	(\mu_{\omega^-} \times \mu_{\omega^-})\left(\partial D_i\right) = 0 \quad {\mathbf{P}}\text{-a.s.}
\end{align}
We set for $i\geq 1$
\begin{align*}
	H_i(x,y,\omega):=\chi_{D_i}(x,y)\liminf_{n\rightarrow \infty} G\left(\Phi^{n}_{\omega}(x),\Phi^{n}_{\omega}(y)\right).
\end{align*}
Note that $R^{\nu}(x, \omega) \geq \frac{1}{n}$. Thus, for every $(x, \omega) \in \mathcal{M}$, it follows from the second assertion of Theorem~\ref{ASASiid}  and~\eqref{GBAss} that
\begin{align}\label{5sadfsr}
	H_i(x,y,\omega)=
	\begin{cases}
		\begin{split}
			& \infty, \quad \text{if} \quad x, y \in D_i,\\
			&0, \quad \text{otherwise}.
		\end{split}
	\end{cases}
\end{align}
From \eqref{GABSs}, \eqref{sddsssw}, Fatou's Lemma and the monotone convergence theorem, we deduce that
\begin{align*}
	\sum_{i\geq 1}\int_{\Pi_{\Omega}(\mathcal{M})}\int_{(\mathcal{M}(\omega) \times \mathcal{M}(\omega)) \setminus \Delta} H_{i}(x,y,\omega)\mu_{\omega^-}(\mathrm{d}x) \mu_{\omega^-}(\mathrm{d}y)\ \mathbf{P}\times\mathbf{P}(\mathrm{d}(\omega^-,\omega^+))<\infty.
\end{align*}
From \eqref{5sadfsr} and using again the fact that almost surely $\mu_{\omega^-}$ has no point masses, we conclude that for \( \mathbf{P}\times\mathbf{P} \)-almost every \( \omega \in \Pi_{\Omega}(\mathcal{M}) \), we have
\begin{align}\label{sayds}
	(\mu_{\omega^-}\times \mu_{\omega^-})\left(\left((\mathcal{M}(\omega) \times \mathcal{M}(\omega)) \setminus \Delta\right)\cap D_i\right)=(\mu_{\omega^-}\times \mu_{\omega^-})\left(\left(\mathcal{M}(\omega) \times \mathcal{M}(\omega)\right)\cap D_i\right)=0.
\end{align}
Recalling \eqref{sddsssw} and the fact that $\bigcup_{i \geq 1} D_i = \mathbb{R}^d \times \mathbb{R}^d$, this further implies that 
\begin{align*}
	(\mu_{\omega^-}\times \mu_{\omega^-})\left(\left(\mathcal{M}(\omega) \times \mathcal{M}(\omega)\right)\right)=\mu_{\omega^-}^2(\mathcal{M}(\omega))=0 \quad {\mathbf{P}\times\mathbf{P}}\text{-a.s.}
\end{align*}
which contradicts \eqref{8sdde}. In conclusion, it is not possible for the almost surely constant function \( \mathcal{H} \) to be equal to zero. Therefore, the claim follows from \textbf{Case I}. 
\qed 
\subsection{Random attractors}\label{RAAT}
We now establish the existence of pullback attractors. To this end, we introduce the following definition tailored to the construction of a random dynamical system provided in~\cite{BNGH26}. 
To distinguish between this definition and  Definition~\ref{ATTRACTOR}, we use the term \( \mathrm{P} \)-pullback attractor to emphasize that this depends only on the past of the noise, which is natural since the SDS filters out the future of the noise.
\begin{definition}\label{attractor}  
	A random compact set $\mathcal{A} = \{ \mathcal{A}(\omega^-) \}_{\omega^- \in \mathcal{B}}$ is called a \(\mathrm{P}\)-pullback attractor for the SDS \(\varphi\) if the following conditions are satisfied.
	\begin{enumerate}
		\item For every $\omega^-, \omega^+ \in \mathcal{B}$, we have
		\[
		\Phi^{t}_{(\omega^-,\omega^+)}(\mathcal{A}(\omega^-)) = \mathcal{A}(P_t(\omega^-,\omega^+)).
		\]
		
		\item For every compact set 
		\( K \subset \mathbb{R}^d \)
		\[
		\lim_{t \to \infty} \sup_{x \in K} d(\varphi^{t}_{\omega^-}(x), {\mathcal{A}}(\omega^-)) = 0 \quad \mathbf{P}\text{-a.s.}
		\]
	\end{enumerate}
\end{definition}
\begin{remark}
	\begin{enumerate}
		\item Similar to Definition~\ref{ATTRACTOR}, other notions such as \(\mathrm{P}\)-weak attractor or \(\mathrm{P}\)-weak point attractor can be analogously defined.
		\item 
		Note that the definition of the \( \mathrm{P} \)-pullback attractor is essentially the same as the definition of a pullback attractor given in Definition~\ref{ATTRACTOR}. 
		In particular, one can see that
		\begin{align}\label{ABNYAUse}
			\Phi^{t}_{\theta_{-t}(\omega^{-}, \omega^{+})}(x)
			= \varphi^{t}_{\omega^{-}}(x),
		\end{align}
		meaning that the second property in Definition \ref{attractor} is the pullback attraction as stated in Definition \ref{ATTRACTOR}.
		However, the main motivation for using this definition is the fact that the pullback attractor \( \mathcal{A}(\omega^{-}, \omega^{+}) \) depends only on \( \omega^{-} \), in the spirit of \cite[Corollary~4.5]{HDS09}.  
		Thus, thanks to \eqref{ABNYAUse} and to the fact that the SDS filters out the future of the noise, it is natural to formulate the invariance property in Definition~\ref{attractor} in terms of the SDS and not of the RDS as in Definition \ref{ATTRACTOR}.
		\item For the uniqueness of attractors, we refer to \cite[Section~5]{C99}, in particular \cite[Corollary~5.8]{C99}.
	\end{enumerate}
\end{remark}
We can prove the following statement. 
\begin{theorem}\label{ATTRRP}
	There exists a \(\mathrm{P}\)-pullback attractor \(\mathcal{A}\) for the SDS \(\varphi\) generated by the SDE \eqref{MAIN}.
\end{theorem}
\proof
From Theorem~\ref{RDS_SDS_MEAU}, it follows that for all $t \geq 0$ and $x \in \mathbb{R}^d$, we have
\begin{align*}
	\Phi^{t}_{(\omega^-,\omega^+)}(x)=\int_{0}^{t}F\left(\Phi^{s}_{(\omega^-,\omega^+)}(x)\right)\mathrm{d}s+x-\sigma (	\mathcal{D}_{H}P_{t}(\omega^-,\omega^+))(-t),
\end{align*}
where we recall that $B^{H}_{t}\bigl(P_{t}(\omega^-, \omega^+)\bigr) = - \bigl(\mathcal{D}_{H} P_{t}(\omega^-, \omega^+)\bigr)(-t)$.
We define the stationary fractional Ornstein-Uhlenbeck process
\begin{align} \label{stationary}
	Z_{t}(\omega):=\sigma\int_{-\infty}^{0}\exp(-(t-\tau))\mathrm{d}(\mathcal{D}_{H}\omega^-)(\tau)+\sigma\int_{0}^{t}\exp(-(t-\tau))\mathrm{d}B^{H}_{\tau}(P_{\tau}(\omega^-,\omega^+)).
\end{align}
To see that this is well-defined as a Young integral on a fixed time interval, we use Definition~\ref{BB_HH} together with Lemma~\ref{AJSJAsd} which ensure that the mapping  
\[
\tau \mapsto B^{H}_{\tau}\bigl(P_{\tau}(\omega^-, \omega^+)\bigr)
\]  
is \(\frac{H}{2}\)-H\"older continuous. 
In~\eqref{TRANST}, we show that the improper integral in~\eqref{stationary} is well defined. Keeping these facts in mind, it follows from~\eqref{stationary} that $Z$ is the stationary solution of the SDE
\begin{align*}
	\mathrm{d}Z_t(\omega) = -Z_t(\omega)\,\mathrm{d}t + \sigma\,\mathrm{d}B^H_t\big(P_t(\omega^-, \omega^+)\big).
\end{align*}
Consequently, for every $0 \leq s \leq t$
\begin{align*}
	Z_t(\omega) - Z_s(\omega) = -\int_s^t Z_\tau(\omega)\,\mathrm{d}\tau + \sigma\left( B^H_t\left(P_{t}(\omega^-,\omega^+)\right) - B^H_s\left(P_{s}(\omega^-,\omega^+)\right)\right).
\end{align*}
Recall that for $t \geq 0$, we have $\theta_{-t}\omega = \left(\vartheta_t \omega^-, P_t(\omega^+, \omega^-)\right)$. Thus, from \eqref{RR_TT1} and \hyperref[IV]{$(iv)$} in Definition~\ref{SDS}, it follows that for all $0 \leq \tau \leq t$
\begin{align*}
	\begin{split}
		&B^{H}_{\tau}\left(P_{\tau}\left(\vartheta_t \omega^-, P_t(\omega^+, \omega^-)\right)\right)=- \left(\mathcal{D}_{H}P_{\tau}\left(\vartheta_t \omega^-, P_t(\omega^+, \omega^-)\right)\right)(-\tau)\\&=\mathcal{D}_{H} P_{t}\left(\vartheta_t \omega^-, P_t(\omega^+, \omega^-)\right)(\tau - t) - \mathcal{D}_{H} P_{t}\left(\vartheta_t \omega^-, P_t(\omega^+, \omega^-)\right)(-t)\\&=(\mathcal{D}_{H}\omega^-)(\tau-t)-(\mathcal{D}_{H}\omega^-)(-t).
	\end{split}
\end{align*}
This further implies using \eqref{RR_TT2} that
\begin{align}
	Z_{s}(\theta_{-t}\omega)&=\sigma\int_{-\infty}^{0}\exp(-(s-\tau))\mathrm{d}(\mathcal{D}_{H}\vartheta_t\omega^-)(\tau)\nonumber\\&+\sigma\int_{0}^{s}\exp(-(s-\tau))\mathrm{d}B^{H}_{\tau}\left(P_{\tau}(\vartheta_t \omega^-, P_t(\omega^+, \omega^-))\right)=\sigma\int_{-\infty}^{s-t}\exp(t-s+\tau)\mathrm{d}(\mathcal{D}_{H}\omega^-)(\tau) \nonumber\\&=\sigma\int_{-\infty}^{0}\exp(\tau)\mathrm{d}(\mathcal{D}_{H}\vartheta_{t-s}\omega^-)(\tau)=:\tilde{Z}_{s-t}(\omega^-).\label{tilde}
\end{align}
which shows the stationarity of $Z$. 
From \eqref{norm:BH}, \eqref{RR_TT2}, and Lemma~\ref{AJSJAsd}, we infer that for every \(n\in\mathbb{N}_0\)
\begin{align*}
	\|\mathcal{D}_H \vartheta_{t-s}\omega^-\|_{\frac H2;[-n-1,-n]}
	&= \sup_{u,v\in[-n-1,-n]} \frac{\bigl|\mathcal{D}_H \omega^-(s-t+u) - \mathcal{D}_H \omega^-(s-t+v)\bigr|}{|u-v|^{\frac H2}} \\
	&\lesssim \sqrt{1 + n + t - s}\,\|\mathcal{D}_H \omega^-\|_{\mathcal{B}_{1-H}}.
\end{align*}
Hence, we can estimate the Young integral by
\begin{align}\label{TRANST}
	\begin{split}
		\left|\int_{-\infty}^0 e^{\tau}\,\mathrm{d}\bigl(\mathcal{D}_H \vartheta_{t-s}\omega^-\bigr)(\tau)\right|
		&\le \sum_{n=0}^\infty \left|\int_{-n-1}^{-n} e^{\tau}\,\mathrm{d}\bigl(\mathcal{D}_H \vartheta_{t-s}\omega^-\bigr)(\tau)\right| \\
		&\lesssim \|\mathcal{D}_H \omega^-\|_{\mathcal{B}_{1-H}} \sum_{n=0}^\infty e^{-n}\sqrt{1 + n + t - s}.
	\end{split}
\end{align}

Now for $\kappa_t(\omega, x) := \Phi^t_{(\omega^-, \omega^+)}(x) - Z_t(\omega)$, we have
\begin{align}\label{ATTRRATT}
	\begin{split}
		\frac{\mathrm{d}}{\mathrm{d} t}\left|\kappa_t(\omega, x)\right|^2 &= 2\left\langle F\left(\kappa_t(\omega, x)+Z_{t}(\omega)\right) - F(Z_{t}(\omega)), \kappa_t(\omega, x) \right\rangle \\
		& + 2\left\langle Z_{t}(\omega) + F(Z_{t}(\omega)), \kappa_t(\omega, x) \right\rangle \\
		&\leq 2C_{1}^{F}-2C_{2}^{F}\left|\kappa_t(\omega, x)\right|^2+\frac{1}{C_2^F} \left|Z_{t}(\omega) + F(Z_{t}(\omega))\right|^2+C_2^F \left|\kappa_t(\omega, x)\right|^2\\
		&= -C_2^F \left|\kappa_t(\omega, x)\right|^2 +2C_1^F + \frac{1}{C_2^F} \left|Z_{t}(\omega) + F(Z_{t}(\omega))\right|^2\\
		&\leq -C_2^F \left|\kappa_t(\omega, x)\right|^2 + \bar{C}^F (1 + |Z_{t}(\omega)|)^{2N},
	\end{split}
\end{align}
for an arbitrary constant \(\bar{C}^F\). Hence, by Gr\"onwall's inequality
\begin{align*}
	\left|\kappa_t(\omega, x)\right|^2\leq \exp(-C_{2}^{F}t)\vert x-Z_{0}(\omega)\vert^2+\bar{C}^F\int_{0}^{t}\exp\left(-C_{2}^{F}(t-\tau)\right)\left(1 + |Z_{\tau}(\omega)|\right)^{2N}\mathrm{d}\tau.
\end{align*}
Thus replacing $\omega$ by \( \theta_{-t}\omega \), we obtain based on~\eqref{tilde} 
\begin{align}\label{YHAs96}
	\begin{split}
		&\left\vert\phi^{t}_{\omega^-}(x)\right\vert\leq \exp(-\frac{C_{2}^{F}}{2}t)\left(\vert x\vert+\vert Z_{0}(\theta_{-t}\omega)\vert\right)\\&+\vert Z_{t}(\theta_{-t}\omega)\vert+\sqrt{\bar{C}^F}\sqrt{\int_{0}^{t}\exp\left(-C_{2}^{F}(t-\tau)\right)\left(1 + |Z_{\tau}(\theta_{-t}\omega)|\right)^{2N}\mathrm{d}\tau}\\
		& = \exp(-\frac{C_{2}^{F}}{2}t)\left(\vert x\vert+\vert \tilde{Z}_{-t}(\omega^-)\vert\right)+\vert \tilde{Z}_{0}(\omega^-)\vert+\sqrt{\bar{C}^F}\sqrt{\int_{-t}^{0}\exp\left(C_{2}^{F}\tau\right)\left(1 + |\tilde{Z}_{\tau}(\omega^-)|\right)^{2N}\mathrm{d}\tau}\\
		&\leq \exp(-\frac{C_{2}^{F}}{2}t)\left(\vert x\vert+\vert Z_{-t}(\omega^-)\vert\right)+\vert \tilde{Z}_{0}(\omega^-)\vert+\sqrt{\bar{C}^F}\sqrt{\int_{-\infty}^{0}\exp\left(C_{2}^{F}\tau\right)\left(1 + |\tilde{Z}_{\tau}(\omega^-)|\right)^{2N}\mathrm{d}\tau},
	\end{split}
\end{align}
where we used the following elementary inequalities 
\begin{align*}
	\lvert X \rvert - \lvert Y \rvert \leq \lvert X - Y \rvert \leq \lvert X \rvert + \lvert Y \rvert, \quad \text{and} \quad \sqrt{\lvert X \rvert + \lvert Y \rvert} \leq \sqrt{\lvert X \rvert} + \sqrt{\lvert Y \rvert}.
\end{align*}
We further set 
\begin{align*}
	\tilde{\rho}(\omega)=\rho(\omega^-):=\vert \tilde{Z}_{0}(\omega^-)\vert+\sqrt{\int_{-\infty}^{0}\bar{C}^F\exp\left(C_{2}^{F}\tau\right)\left(1 + |\tilde{Z}_{\tau}(\omega^-)|\right)^{2N}\mathrm{d}\tau}.
\end{align*}
Note that $\tilde{\rho}(\omega)$ is tempered. To see this, by Remark~\ref{TEMOER}, it suffices to verify that  
\[
\E\!\left[\sup_{t \in [0,1]} \tilde{\rho}(\theta_t \omega)\right] < \infty.
\]
This follows directly from~\eqref{TRANST} together with Fernique's theorem~\cite[Theorem~4]{Led96}.
Since for the first term in~\eqref{YHAs96} we have for every $\omega^- \in \mathcal{B}$ that
\begin{align*}
	\lim_{t\rightarrow \infty} \exp(-\frac{C_{2}^{F}}{2}t)\left(\vert x\vert+\vert \tilde{Z}_{-t}(\omega^-)\vert\right) =0,
\end{align*}
we get that
\( B(0, \epsilon + \rho(\omega^-) ) \)  is a random absorbing set for $\phi$. Then according to~\cite[Theorem~3.5]{FS96}, the pullback attractor is given by 
\begin{align}\label{ASSQ5sd}
	\cA(\omega^-)=\bigcap_{s\geq 0}\bigcup_{t\geq s}\Phi^{t}_{\theta_{-t}\omega}\left(B(0, \epsilon + \tilde{\rho}(\theta_{-t}\omega) )\right)=\bigcap_{s\geq 0}\bigcup_{t\geq s}\phi^{t}_{\omega^-}\left(B(0, \epsilon + \rho(\vartheta_{t}\omega^-) )\right).
\end{align}
\qed
\begin{remark}
	It is straightforward to verify that the collection of tempered sets $\mathcal{D}(\mathbb{R}^d)$ is inclusion-closed \cite[Definition~3.2]{FS96}. Consequently, for the $\mathrm{P}$-pullback attractor established in Theorem~\ref{ATTRRP}, we conclude that it attracts tempered sets and not only deterministic compact sets.
\end{remark}
\section{Synchronization by noise}\label{sec:main}
We are now ready to present the main result of this manuscript. We first show that the support of the disintegrated measures $(\mu_{\omega^{-}})_{\omega^{-}\in \tilde{B}}$ is contained in any $\mathrm{P}$-weak point attractor. We then explore this relation further with the minimal $\mathrm{P}$-weak point attractor. Finally, we prove weak synchronization under certain conditions on $F$ and $\sigma$.
\subsection{Support of the invariant measure of the RDS}\label{SUPPAS}
We begin with the following useful lemma, which relates time averages to ensemble averages in our setting.
\begin{lemma}\label{YAs85sad}
	Let $x_0\in\R^d$ and $G\in  \mathrm{BL}_{\mathcal{B}}(\mathbb{R}^d)$. Then
	\[
	\lim_{t\rightarrow\infty}\frac{1}{t}\int_{0}^{t}\int_{\mathcal{B}}G(\varphi^{s}_{\omega^-}(x_0),\omega^-)\ \mathbf{P}(\mathrm{d\omega^-})\mathrm{d}s=\int_{\R^d\times \mathcal{B}}G(x,\omega^-)\mu\left(\mathrm{d}(x,\omega^-)\right).
	\]
\end{lemma}
\proof
We prove this by contradiction. Assume that for some fixed $x_0 \in \mathbb{R}^d$ and $\tilde{G} \in \mathrm{BL}_{\mathcal{B}}(\mathbb{R}^d)$, there exists an increasing sequence $\{t_n\}_{n \geq 0}$ converging to infinity and a constant $\delta > 0$ such that for every $t_n$ 
\begin{align}\label{AOOS}
	\left| \frac{1}{t_n} \int_{0}^{t_n} \int_{\mathcal{B}} \tilde{G}(\varphi^{s}_{\omega^-}(x_0), \omega^-) \, \mathbf{P}(\mathrm{d}\omega^-) \, \mathrm{d}s 
	- \int_{\mathbb{R}^d \times \mathcal{B}} \tilde{G}(x, \omega^-) \, \mu(\mathrm{d}(x, \omega^-)) \right| > \delta.
\end{align}
From \eqref{TOTAL}, we have for $t\geq 0$
\begin{align*}
	\mathcal{L}\left(\Phi^{t}_{(\omega^-,\omega^+)}(x_0)\right) \xrightarrow{t \to \infty} \left(\Pi_{\mathbb{R}^d}\right)_{\star} \mu,
\end{align*}
in total variation. Since for every \( t \geq 0 \), the random variables \( \Phi^{t}_{(\omega^-,\omega^+)}(x_0) \) and \( \Phi^{t}_{\theta_{-t}(\omega^-,\omega^+)}(x_0) = \varphi^{t}_{\omega^-}(x_0) \) have the same law, it holds that 
\begin{align}\label{YH5286as}
	\mathcal{L}\left(\varphi^{t}_{\omega^-}(x_0)\right) \xrightarrow{t \to \infty} \left(\Pi_{\mathbb{R}^d}\right)_{\star} \mu,
\end{align}
in total variation. Let \( \{ \nu_{t_n, x_0} \}_{t_n \geq 0} \) be a family of measures on \( \mathbb{R}^d \times \mathcal{B} \) defined by 
\begin{align*}
	\begin{split}
		\nu_{t_n,x_0}(\Gamma) := \frac{1}{t_n} \int_{0}^{t_n} \int_{\mathcal{B} \times \mathcal{B}} \chi_{\Gamma} \left( \varphi^{s}_{P_{s}(\omega^-,\omega^+)}(x_0), P_{s}(\omega^-,\omega^+) \right) \mathbb{P}(\mathrm{d}(\omega^-,\omega^+)) \, \mathrm{d}s \\
		= \frac{1}{t_n} \int_{0}^{t_n} \int_{\mathcal{B}} \chi_{\Gamma}(\varphi^{s}_{\omega^-}(x_0), \omega^-) \, \mathbf{P}(\mathrm{d}\omega^-) \, \mathrm{d}s,
	\end{split}
\end{align*}
for \( \Gamma \in \sigma(\mathbb{R}^d) \otimes \sigma(\mathcal{B}) \). Equivalently, this can be written using the transition semigroup $\{\cQ_t\}_{t>0}$ as
\begin{align*}
	\nu_{t_n,x_0}(\Gamma) = \frac{1}{t_n} \int_{0}^{t_n} \int_{\mathcal{B}} \mathcal{Q}_{s}\left((x_0,\omega^-); \Gamma \right) \, \mathbf{P}(\mathrm{d}\omega^-) \, \mathrm{d}s.
\end{align*}
From Definition \ref{DEAFRS}, we conclude that for every \( t \geq 0 \),
\begin{align}\label{OLAS5sdwa}
	\begin{split}
		(\mathcal{Q}_t)_{\star} \nu_{t_n,x_0}(\Gamma) &= \frac{1}{t_n} \int_{0}^{t_n} \int_{\mathcal{B}} \mathcal{Q}_{t+s}\left((x_0,\omega^-); \Gamma \right) \, \mathbf{P}(\mathrm{d}\omega^-) \, \mathrm{d}s \\
		&= \nu_{t_n,x_0}(\Gamma) + \frac{1}{t_n} \int_{t_n}^{t_n + t} \int_{\mathcal{B}} \mathcal{Q}_{s}\left((x_0,\omega^-); \Gamma \right) \, \mathbf{P}(\mathrm{d}\omega^-) \, \mathrm{d}s \\
		&\quad - \frac{1}{t_n} \int_{0}^{t} \int_{\mathcal{B}} \mathcal{Q}_{s}\left((x_0,\omega^-); \Gamma \right) \, \mathbf{P}(\mathrm{d}\omega^-) \, \mathrm{d}s.
	\end{split}
\end{align}
Note that \( \left(\Pi_{\mathbb{R}^d}\right)_{\star} \mu \) is a probability measure and hence regular. Thus, from \eqref{YH5286as}, we deduce that for every \( \epsilon > 0 \), there exists a compact subset \( K_{\epsilon} \subset \mathbb{R}^d \) and a time \( t_{\epsilon} > 0 \) such that for all \( t \geq t_{\epsilon} \),
\begin{align}\label{YH986a}
	\nu_{t,x_0}(K_{\epsilon} \times \mathcal{B}) > 1 - \epsilon.
\end{align}
This implies that the family \( \{ \nu_{t_n, x_0} \}_{n \geq 0} \) is tight in \( \mathrm{Pr}_{\mathcal{B}}(\mathbb{R}^d) \) according to \cite[Definition 4.2]{CH02}. By Prokhorov's theorem for random measures \cite[Theorem 4.4]{CH02}, it follows that this family is relatively compact with respect to the narrow topology on \( \mathrm{Pr}_{\mathcal{B}}(\mathbb{R}^d) \) which 
by \cite[Corollary 4.10]{CH02} is generated by the class of mappings
\begin{align*}
	\mathrm{Pr}_{\mathcal{B}}(\mathbb{R}^d) &\longrightarrow \mathbb{R}, \\
	\nu &\mapsto \int_{\mathbb{R}^d \times \mathcal{B}} G(x, \omega^-) \, \nu(\mathrm{d}(x, \omega^-)),
\end{align*}
for all \( G \in \mathrm{BL}_{\mathcal{B}}(\mathbb{R}^d) \).
Let \( \nu \) be a limit point of \( \{ \nu_{t_n, x_0} \} \). From \eqref{OLAS5sdwa}, it follows that \( \nu \) is invariant under the transition semigroup \( \{ \mathcal{Q}_t \}_{t > 0} \). Therefore, by the uniqueness of the invariant measure of the SDS $\mu$, we conclude that \( \nu = \mu \).
Since \( \mu \) is the unique limit point of the family \( \{ \nu_{t_n, x_0} \} \) with respect to the narrow topology, we deduce that for a subsequence \( \{ t_{n_k} \}_{k \geq 0} \subseteq \{ t_n \}_{n \geq 0} \),
\begin{align*}
	\lim_{t_{n_k} \to \infty}\int_{\R^d\times \mathcal{B}}\tilde{G}(x,\omega^-)\nu_{t_{n_k},x_0}\left(\mathrm{d}(x,\omega^-)\right) &=	\lim_{t_{n_k} \to \infty} \frac{1}{t_{n_k}} \int_{0}^{t_{n_k}} \int_{\mathcal{B}} \tilde{G}(\varphi^{s}_{\omega^-}(x_0), \omega^-) \, \mathbf{P}(\mathrm{d}\omega^-) \, \mathrm{d}s 
	\\&= \int_{\mathbb{R}^d \times \mathcal{B}} \tilde{G}(x, \omega^-) \, \mu(\mathrm{d}(x, \omega^-)).
\end{align*}
This contradicts \eqref{AOOS}.
\qed
\begin{remark}\label{RRFFDASd}
	Note that time averaging is a standard technique used to ensure that the family of tight measures \( \nu_{t_n,x_0} \) converges to \( \mu \). Although the family of measures
	\[
	\int_{\mathcal{B}} \mathcal{Q}_{s}\left((x_0,\omega^-); \Gamma \right) \, \mathbf{P}(\mathrm{d}\omega^-), \quad s \geq 0,
	\]
	is tight, it is not evident that \( \mu \) arises as a limit point of this family.
	For a strongly mixing, white-noise random dynamical system this holds true, see \cite{KS04} for details.
	We expect that an analogous result should be valid in our setting as well, although the proof seems challenging in the absence of the Markov property.
\end{remark}
\begin{proposition}\label{YASHAsd}
	Let \( \mathcal{A} \) be a \( \mathrm{P} \)-weak point attractor. Then
	\begin{align}\label{AUSJl63a}
		\mu_{\omega^-}\big({\mathcal{A}}(\omega^-)\big) = 1  \quad \mathbf{P}\text{-a.s.},
	\end{align}
	which in particular implies that
	\begin{align}\label{asdf32as}
		\mathrm{supp}(\mu_{\omega^-})\subseteq {\mathcal{A}}(\omega^-)  \quad \mathbf{P}\text{-a.s.}
	\end{align}
\end{proposition}
\proof
For an arbitrary $\epsilon > 0$ we define
\begin{align}\label{GG_EE}
	\begin{split}
		G_{\epsilon} \colon \R^d \times \tilde{\mathcal{B}} &\longrightarrow [0,1],\\
		(x, \omega^-) &\longrightarrow 1 - \min\left\lbrace \frac{d(x, \mathcal{A}(\omega^-))}{\epsilon}, 1 \right\rbrace.
	\end{split}
\end{align}
Here we recall that $\tilde{\mathcal{B}}$ is the measurable subset of $\mathcal{B}$ obtained in Lemma~\ref{RAND_MEASURE}. It is straightforward to verify that \( G_{\epsilon} \in \mathrm{BL}_{\mathcal{B}}(\mathbb{R}^d) \). Hence, by Lemma~\ref{YAs85sad}, we obtain
\begin{align}\label{Ola63sfsa}
	\lim_{t \rightarrow \infty} \frac{1}{t} \int_{0}^{t} \int_{\mathcal{B}} G_{\epsilon}(\varphi^{s}_{\omega^-}(x_0), \omega^-) \, \mathbf{P}(\mathrm{d} \omega^-) \, \mathrm{d}s 
	= \int_{\mathbb{R}^d \times \mathcal{B}} G_{\epsilon}(x, \omega^-) \, \mu\left( \mathrm{d}(x, \omega^-) \right).
\end{align}
Let \( 0 < \delta < 1 \) be arbitrarily small. Then it is clear that
\begin{align}\label{OALssdvz}
	(1 - \delta)\chi_{\mathcal{N}(\omega^-, \epsilon \delta)}(x) \leq G_{\epsilon}(x, \omega^-) \leq \chi_{\mathcal{N}(\omega^-, 2\epsilon)}(x),
\end{align}
where \( \mathcal{N}(\omega^-, \alpha) = \left\lbrace x \in \mathbb{R}^d : d(x, \mathcal{A}(\omega^-)) < \alpha \right\rbrace \) for \( \alpha > 0 \).  
Therefore, combining \eqref{Ola63sfsa} and \eqref{OALssdvz} and using that $\mathcal{A}$ is a $\mathrm{P}$-weak point attractor, we obtain for $x_0 \in \mathbb{R}^d$ that 
\begin{align}\label{KJ145}
	\begin{split}   
		(1 - \delta) 
		&= (1 - \delta) \lim_{t \rightarrow \infty} \frac{1}{t} \int_{0}^{t} \left(1 - \mathbf{P}\left(d\left(\varphi^{s}_{\omega^-}(x_0), \mathcal{A}(\omega^-)\right) \geq \epsilon \delta\right)\right) \mathrm{d}s \\
		&\leq \int_{\tilde{\mathcal{B}}} \mu_{\omega^-} \left( \mathcal{N}(\omega^-, 2\epsilon) \right) \, \mathbf{P}(\mathrm{d} \omega^-) \leq 1.
	\end{split}
\end{align}
Thus, since \( 0 < \delta < 1 \) can be chosen arbitrarily small, we obtain
\[
\int_{\tilde{\mathcal{B}}} \mu_{\omega^-} \left( \mathcal{N}(\omega^-, 2\epsilon) \right) \, \mathbf{P}(\mathrm{d} \omega^-) = 1.
\]
Obviously \( \mu_{\omega^-} \left( \mathcal{N}(\omega^-, 2\epsilon) \right) \leq 1 \) and $\mathbf{P}(\tilde{B})=1$. Hence 
\begin{align}\label{HAKsqq}
	\mu_{\omega^-} \left( \mathcal{N}(\omega^-, 2\epsilon) \right) = 1, \quad \mathbf{P}\text{-a.s.}
\end{align}
Since \( \epsilon > 0 \) is arbitrary and using the fact that \( \mathcal{A}(\omega^-) \) is closed, it follows that
\begin{align}\label{asdadcswf}
	\mathcal{A}(\omega^-) = \bigcap_{n \geq 1} \mathcal{N} \left( \omega^-, \tfrac{1}{n} \right).
\end{align}
This together with \eqref{HAKsqq} implies \eqref{AUSJl63a}. 
\qed

\begin{remark} 
	From the definition, the $\mathrm{P}$-pullback attractor is also a $\mathrm{P}$-weak point attractor. Therefore, the claim of Proposition \ref{JHJHJH} is in particular valid for the $\mathrm{P}$-pullback attractor.~However, since we plan to show a weak type of synchronization, we further focus on weak point attractors.
\end{remark}

Let us continue with the following lemma.
\begin{lemma}\label{ILAS9sdzw}
	Let $\epsilon > 0$ and define
	\[
	G_{\epsilon} : \mathbb{R}^d \times \mathcal{B} \longrightarrow [0,1]
	\]
	by
	\[
	G_\epsilon(x, \omega^-): =
	\begin{cases}
		\min\left\{ \dfrac{d\left(x,\mathrm{supp}(\mu_{\omega^-})\right)}{\epsilon}, 1 \right\}, & \text{if } (x, \omega^-) \in \mathbb{R}^d \times \tilde{\mathcal{B}}, \\
		0, & \text{otherwise}.
	\end{cases}
	\]
	Then \( G_\epsilon \in \mathrm{BL}_{\mathcal{B}}(\mathbb{R}^d) \) and is jointly measurable.
\end{lemma}
\proof
First, it is not hard to see that for every $x, y \in \mathbb{R}^d$, we have
\[
|G_\epsilon(x, \omega^-) - G_\epsilon(y, \omega^-)| \leq \frac{\lvert x - y \rvert}{\epsilon}.
\]
To prove the measurability of $G_\epsilon$ with respect to $\omega^-$, it suffices to show that for every \( r > 0 \),
\[
\left\{ \omega^- \in \tilde{\mathcal{B}} : d\left(x, \mathrm{supp}(\mu_{\omega^-})\right) < r \right\} \in \sigma(\mathcal{B}).
\]
To see this, note that by the definition of the support of a measure and Lemma~\ref{RAND_MEASURE}, we have
\[
\left\{ \omega^- \in \tilde{\mathcal{B}} : d\left(x,\mathrm{supp}(\mu_{\omega^-})\right) < r \right\} = \left\{ \omega^- \in \tilde{\mathcal{B}} : \mu_{\omega^-}(B(x, r)) > 0 \right\} \in \sigma(\mathcal{B}).
\]
Therefore, \( G_\epsilon \in \mathrm{BL}_{\mathcal{B}}(\mathbb{R}^d) \). The joint measurability of $G_\epsilon$ follows from \cite[Lemma 1.1]{CH02}.
\qed
\begin{remark}\label{UJASS} 
	From Proposition~\ref{YASHAsd}, for every \(\mathrm{P}\)-weak point attractor \(\mathcal{A}\), it follows that
	\[
	\mathrm{supp}(\mu_{\omega^-}) \subseteq \mathcal{A}(\omega^-).
	\]
	A natural question is whether the family 
	\(\{ \mathrm{supp}(\mu_{\omega^-}) \}_{\omega^- \in \tilde{\mathcal{B}}}\) 
	itself constitutes a \(\mathrm{P}\)-weak point attractor. 
	According to Remark~\ref{ARATSasd}, once we have a strongly mixing white-noise RDS, this property holds.
\end{remark}
The following proposition provides a partial answer to this question.
\begin{proposition}\label{POINT}
	Let \( \overline{\mathcal{A}} = \left\{ \overline{\mathcal{A}}(\omega^-) \right\}_{\omega^- \in \mathcal{B}} \) be defined by
	\[
	\overline{\mathcal{A}}(\omega^-) :=
	\begin{cases}
		\mathrm{supp}(\mu_{\omega^-}) & \text{if } \omega^- \in \tilde{\mathcal{B}}, \\
		\{0\} & \text{otherwise}.
	\end{cases}	
	\]
	Then, for every $\epsilon>0$ and $x_{0}\in\R^d$
	\begin{align}\label{ASOASAsw}
		\lim_{t\rightarrow \infty}\frac{1}{t}\int_{0}^{t}\mathbf{P}\left(d\left(\varphi^{s}_{\omega^-}(x_0),\overline{\mathcal{A}}(\omega^-)\right)\geq\epsilon\right)\mathrm{d}s=0.
	\end{align}
\end{proposition}
\proof
Let $G_\epsilon$ be as in Lemma \ref{ILAS9sdzw}. From Lemma~\ref{YAs85sad}, we have
\begin{align}
	\begin{split}
		&\lim_{t\rightarrow\infty}\frac{1}{t}\int_{0}^{t}\int_{\mathcal{B}}G_{\epsilon}(\varphi^{s}_{\omega^-}(x_0),\omega^-)\ \mathbf{P}(\mathrm{d\omega^-})\mathrm{d}s =\int_{\R^d\times \mathcal{B}}G_{\epsilon}(x,\omega^-)\mu\left(\mathrm{d}(x,\omega^-)\right)\\
		&=\int_{\tilde{\mathcal{B}}}\int_{\mathrm{supp}(\mu_{\omega^-})}G_{\epsilon}(x,\omega^-)\mu_{\omega^-}(\mathrm{d}x)\mathbf{P}(\mathrm{d}\omega^-)=0.
	\end{split}
\end{align}
Thus, from the definition of $G_\epsilon$ 
\begin{align*}
	&\lim_{t\rightarrow \infty}\frac{1}{t}\int_{0}^{t}\mathbf{P}\left(d\left(\varphi^{s}_{\omega^-}(x_0),\overline{\mathcal{A}}(\omega^-)\right)\geq\epsilon\right)\mathrm{d}s\\&\quad \leq \lim\limits_{t\to\infty} \frac{1}{t}\int_{0}^{t}\int_{\mathcal{B}}G_{\epsilon}(\varphi^{s}_{\omega^-}(x_0),\omega^-)\ \mathbf{P}(\mathrm{d\omega^-})\mathrm{d}s=0.
\end{align*}
This concludes the proof.
\qed
\begin{remark}\label{AASASA}
	Unfortunately, Proposition~\ref{POINT} does not allow us to conclude that \(\overline{\mathcal{A}}\) is a \(\mathrm{P}\)-weak point attractor. Nevertheless, it is evident that every \(\mathrm{P}\)-weak point attractor satisfies the condition \eqref{ASOASAsw}. This observation motivates the conjecture that \(\overline{\mathcal{A}}\) is a \(\mathrm{P}\)-weak point attractor and hence minimal.
\end{remark}
Let us now examine how Proposition~\ref{POINT} provides a partial answer to the fact that $\overline{\mathcal{A}}$ is the minimal \( \mathrm{P} \)-weak point attractor.
\begin{proposition}\label{AUIASLc}
	Consider the setting of Proposition~\ref{POINT}. Let \( (x_i)_{i \geq 1} \) be an arbitrary sequence in \( \mathbb{R}^d \). Then there exists an increasing sequence \( (s_n)_{n \geq 1} \) with \( s_n \to \infty \) as \( n \to \infty \), such that
	\begin{align*}
		\lim_{s_n \rightarrow \infty} \mathbf{P}\left(d\left(\varphi^{s_n}_{\omega^-}(x_i), \overline{\mathcal{A}}(\omega^-)\right) \geq t \right) = 0,
	\end{align*}
	for all \( i \in \mathbb{N} \) and \( t > 0 \).
\end{proposition}
\proof
First, note that Proposition~\ref{POINT} implies that for 
\[
\mathcal{R}_{m}(s,\omega^-) := \sum_{\substack{1 \leq j,k \leq m }} \mathbf{P}\left( d\big(\varphi^{s}_{\omega^-}(x_{k}), \overline{\mathcal{A}}(\omega^-)\big) \geq \frac{1}{j} \right)
\]
we have
\[
\lim_{t\rightarrow\infty}\frac{1}{t}\int_{0}^{t}\mathcal{R}_{m}(s,\omega^-)\,\mathrm{d}s = 0.
\]
Therefore, we can find a sequence \( (s_n^{m})_{n \geq 1} \) such that
\[
\lim_{s_n^{m} \rightarrow \infty} \mathcal{R}_{m}(s_n^m,\omega^-) = 0.
\]
Now, applying a standard diagonal argument to the sequence of increasing functions \( (\mathcal{R}_{m})_{m \geq 1} \) completes the proof.
\qed
We further investigate the structure of $\overline{\cA}$. 
\begin{lemma}\label{JHJHJH}
	Consider the setting of Proposition~\ref{POINT}. Then there exists a deterministic constant \( \overline{C} > 0 \) such that
	\begin{align*}
		\int_{\tilde{\mathcal{B}}} \left(\mu_{\omega^-} \times \mu_{\omega^-}\right)(\Gamma_{\overline{C}}) \, \mathbf{P}(\mathrm{d}\omega^-)=0,
	\end{align*}
	where
	\[
	\Gamma_{\overline{C}} := \left\{ (x, y) \in \mathbb{R}^d \times \mathbb{R}^d : |x - y| > \overline{C} \right\}.
	\]
\end{lemma}
\proof
Let \( x, y \in \mathbb{R}^d \) and define
\[
\kappa^{t}_{(\omega^-,\omega^+)}(x,y) := \Phi^t_{(\omega^-, \omega^+)}(x) - \Phi^t_{(\omega^-, \omega^+)}(y).
\]
Then, we have
\begin{align*}
	\frac{\mathrm{d}}{\mathrm{d}t} \left| \kappa^{t}_{(\omega^-,\omega^+)}(x,y) \right|^2 
	&= 2 \left\langle F\left( \Phi^t_{(\omega^-, \omega^+)}(x) \right) - F\left( \Phi^t_{(\omega^-, \omega^+)}(y) \right), \kappa^{t}_{(\omega^-,\omega^+)}(x,y) \right\rangle \\
	&\leq 2C_{1}^{F} - 2C_{2}^{F} \left| \kappa^{t}_{(\omega^-,\omega^+)}(x,y) \right|^2.
\end{align*}
Thus, by Gr\"onwall's inequality,
\begin{align}\label{ASoapwsQ}
	\begin{split}  
		\left| \Phi^t_{(\omega^-, \omega^+)}(x) - \Phi^t_{(\omega^-, \omega^+)}(y) \right|^2 
		&\leq \exp\left( -2C_{2}^{F} t \right) |x - y|^2 
		+ 2C^F_1 \int_0^t \exp\left( -2C_{2}^{F} (t-\tau) \right) \, \mathrm{d}\tau\\&\leq  \exp\left( -2C_{2}^{F} t \right) |x - y|^2 +\frac{C_{1}^{F}}{C_{2}^{F}}.
	\end{split}
\end{align}
Note that for every \( \Gamma \in \sigma(\mathbb{R}^d) \otimes \sigma(\mathbb{R}^d) \), analogously to \eqref{TGB2} and \eqref{GABSs}, we have
\begin{align}\label{aisdfa}
	\begin{split}
		&\int_{\mathcal{B}} \int_{\mathbb{R}^d \times \mathbb{R}^d} \scalebox{1.2}{$\chi_{\Gamma}$}(x, y) \, \mu_{\omega^-}(\mathrm{d}x) \mu_{\omega^-}(\mathrm{d}y) \, \mathbf{P}(\mathrm{d}\omega^-) 
		\\&= \int_{\mathcal{B}} \int_{\tilde{\mathcal{B}}} \int_{\mathbb{R}^d \times \mathbb{R}^d} \scalebox{1.2}{$\chi_{\Gamma}$}\left(\Phi^{t}_{(\omega^-, \omega^+)}(x), \Phi^{t}_{(\omega^-, \omega^+)}(y)\right) \mu_{\omega^-}(\mathrm{d}x) \mu_{\omega^-}(\mathrm{d}y) \, \mathbf{P}(\mathrm{d}\omega^-) \mathbf{P}(\mathrm{d}\omega^+).
	\end{split}
\end{align}
For $\epsilon>0$, setting 
\[
\overline{C}:= \sqrt{\frac{C_{1}^F}{C_{2}^F}}+\epsilon,
\]
it follows from \eqref{ASoapwsQ} and the dominated convergence theorem that
\[
\lim_{t \rightarrow \infty} \int_{\mathcal{B}} \int_{\tilde{\mathcal{B}}} \int_{\mathbb{R}^d \times \mathbb{R}^d} 
\scalebox{1.2}{$\chi_{\Gamma_{\overline{C}}}$}\left(\Phi^{t}_{(\omega^-, \omega^+)}(x), \Phi^{t}_{(\omega^-, \omega^+)}(y)\right) 
\mu_{\omega^-}(\mathrm{d}x) \mu_{\omega^-}(\mathrm{d}y) \, \mathbf{P}(\mathrm{d}\omega^-) \mathbf{P}(\mathrm{d}\omega^+)=0.
\]

This, together with \eqref{aisdfa}, yields that
\begin{align}\label{spd65a}
	\begin{split}
		&\int_{\mathcal{B}} \int_{\mathbb{R}^d \times \mathbb{R}^d} \scalebox{1.2}{$\chi_{\Gamma_{\overline{C}}}$}(x, y) \, \mu_{\omega^-}(\mathrm{d}x) \mu_{\omega^-}(\mathrm{d}y) \, \mathbf{P}(\mathrm{d}\omega^-)\\&\quad=	\int_{\tilde{\mathcal{B}}} \left(\mu_{\omega^-} \times \mu_{\omega^-}\right)(\Gamma_{\overline{C}}) \, \mathbf{P}(\mathrm{d}\omega^-)=0.
	\end{split}
\end{align}
\qed
Now, we can state the following result.
\begin{corollary}\label{ATRTASs}
	Assume the hypotheses of the second statement of 
	Theorem~\ref{ASASiid} together with Proposition~\ref{POINT} hold. Then there exists an integer \( p \) and a collection of random variables \( \{ b_i \}_{1 \leq i \leq p} \), where each \( b_i : \tilde{\mathcal{B}} \rightarrow \mathbb{R}^d \), such that: 
	\begin{itemize}
		\item[1)] The family \( \overline{\mathcal{A}} = \{ \overline{\mathcal{A}}(\omega^-) \}_{\omega^- \in \mathcal{B}} \) is given by
		\[
		\overline{\mathcal{A}}(\omega^-) =
		\begin{cases}
			\left\{ b_i(\omega^-)\right\}_{1 \leq i \leq p}, & \text{if } \omega^- \in \tilde{\mathcal{B}}, \\
			\{0\}, & \text{otherwise}.
		\end{cases}
		\]
		Moreover, for every $\omega^-\in \tilde{\mathcal{B}}$
		\[
		\mu_{\omega^-} = \frac{1}{p} \sum_{i=1}^{p} \delta_{\{ b_i(\omega^-) \}}.
		\]    
		\item[2)] For every \( \epsilon > 0 \) and every \( x_0 \in \mathbb{R}^d \),
		\begin{align*}
			\lim_{t \rightarrow \infty} \frac{1}{t} \int_0^t \mathbf{P}\left( d\left( \varphi^s_{\omega^-}(x_0), \overline{\mathcal{A}}(\omega^-) \right) \geq \epsilon \right) \, \mathrm{d}s = 0.
		\end{align*}
		
		\item[3)] {For the constant $C$ given in Lemma \ref{JHJHJH}, we have } 
		\[
		\max_{1\leq i \neq j\leq p} \left| b_i(\omega^-) - b_j(\omega^-) \right| \leq C, \quad \mathbf{P}\text{-a.s.}
		\]
	\end{itemize}
\end{corollary}
\proof
The first two claims follow from Propositions~\ref{descrte} and \ref{POINT}. Note that the existence of the random variables \( \{ b_i \}_{1 \leq i \leq p} \) follows from any measurable selection theorem applied to the family \( \{ \mathrm{supp}(\mu_{\omega^-}) \}_{\omega^- \in \tilde{\mathcal{B}}} \). Given this collection of random variables, it follows from \eqref{spd65a} that
\begin{align*}
	&\int_{\tilde{\mathcal{B}}} \int_{\mathbb{R}^d \times \mathbb{R}^d} \scalebox{1.2}{$\chi_{\Gamma_C}$}(x, y) \, \mu_{\omega^-}(\mathrm{d}x) \mu_{\omega^-}(\mathrm{d}y) \, \mathbf{P}(\mathrm{d}\omega^-) 
	\\&\quad= \sum_{\substack{1 \leq i, j \leq p \\ i \neq j}} \frac{1}{p^2} \int_{\tilde{\mathcal{B}}}  \scalebox{1.2}{$\chi_{\Gamma_C}$}\big(b_i(\omega^-), b_j(\omega^-)\big) \, \mathbf{P}(\mathrm{d}\omega^-) \\ &\qquad = \sum_{\substack{1 \leq i, j \leq p \\ i \neq j}} \frac{1}{p^2} \mathbf{P}\left(\left| b_i(\omega^-) - b_j(\omega^-) \right| > C\right) = 0,
\end{align*}
which proves the last claim.  
\qed
\subsection{Weak synchronization}
In Subsection~\ref{SUPPAS}, we investigated the relationship between random measures and weak point attractors. In particular, we observed in Propositions~\ref{POINT} and~\ref{AUIASLc} that one can interpret the support of the random measures as minimal \(\mathrm{P}\)-weak point attractors in an averaged sense. 
Moreover, when the top Lyapunov exponent is negative, the support of these random measures is discrete. The question which arises is whether this support consists of a single point, as investigated in 
~\cite{FGS16a} for white-noise RDS.~To address this issue in the non-Markovian case, we first introduce an additional assumption on the drift.
\begin{assumptions}\label{AASQw98a}
	In addition to Assumption~\ref{Drift}, we assume that the drift term \(F\) is eventually strictly monotone. This means that there exist constants \(R > 0\) and \(C_4^F > 0\) such that for all \(x, y \in \mathbb{R}^d\) with \(|x|, |y| \geq R\),
	\begin{align}\label{monotone st11}
		\langle F(x) - F(y), x - y \rangle \leq -C_4^F |x - y|^2.
	\end{align}
\end{assumptions}
We further impose the following conditions:
\begin{assumptions}\label{AASQw98a1}
	\begin{itemize}
		\item[1)] There exists an integer \(p\) and a collection of random variables \(\{ b_i : \tilde{\mathcal{B}} \rightarrow \mathbb{R}^d \}_{1 \leq i \leq p}\) such that for every \(\omega^- \in \tilde{\mathcal{B}}\), the support of \(\mu_{\omega^-}\) is given by
		\[
		\mathrm{supp}(\mu_{\omega^-}) = \left\{ b_i(\omega^-) \mid 1 \leq i \leq p \right\}.
		\]
		\item[2)] There exists a deterministic constant \(C > 0\) and a set 
		\(\tilde{\Omega} \subset \Omega \), \(\tilde{\Omega} := \tilde{\mathcal{B}} \times \mathcal{B}\) of full measure such that for every 
		\(\omega = (\omega^-, \omega^+) \in \tilde{\Omega}\), we have
		\begin{align}\label{YAsjass}
			\mathcal{R}(\omega^-) := \max_{\substack{1 \leq i, j \leq p \\ i \neq j}} \left| b_i(\omega^-) - b_j(\omega^-) \right| \leq C .
		\end{align}
		Furthermore, the set $\tilde{\Omega}$ is invariant under $\theta_{t}$ meaning that  
		\[
		\theta_{t}(\tilde{\Omega}) \subseteq \tilde{\Omega} \text{  for all  } t\geq 0.
		\]
	\end{itemize}  
\end{assumptions}
\begin{remark} 
	The eventually strict monotonicity condition on the drift \(F\) is necessary to show that the top Lyapunov exponent of the SDE~\eqref{MAIN} is negative for increasing \(\sigma\), as established in \cite{BNGH26}. This implies condition~1) of the previous assumption, as shown in Proposition \ref{descrte}. 
\end{remark}
\begin{remark}\label{IAKSOded}
	Condition~2) of Assumption~\ref{AASQw98a1} is implied by the assumptions of Corollary~\ref{ATRTASs}. More precisely, Corollary~\ref{ATRTASs} yields the existence of a subset 
	\(\overline{\mathcal{B}} \subseteq \tilde{\mathcal{B}}\) such that 
	\(\mathbf{P}(\overline{\mathcal{B}}) = 1\) and
	\[
	\max_{i \neq j} \lvert b_i(\omega^-) - b_j(\omega^-) \rvert \leq C,
	\qquad \forall \, \omega^- \in \overline{\mathcal{B}}.
	\]
	By considering
	\[
	\overline{\Omega} := \bigcap_{\substack{k \in \mathbb{Z},\\ j \in \mathbb{N}}} 
	\theta_{\tfrac{k}{2^j}}(\overline{\mathcal{B}} \times \mathcal{B}),
	\]
	we obtain a full measure set that is $\theta_{\tfrac{k}{2^j}}$-invariant for every $j \in \mathbb{N}$ and $k\in\Z$.   
	Now, define 
	\begin{align}\label{tilde:omega}
		\tilde{\Omega} := \bigl\{ \omega \in \overline{\mathcal{B}} \times \mathcal{B} : 
		\forall t \in \mathbb{R}, \; \mathcal{R}(P_t(\omega^-,\omega^+)) \leq C \bigr\}.
	\end{align}
	By definition, $\tilde{\Omega}$ is $\theta_t$-invariant for every $t \in \mathbb{R}$ and $\tilde{\Omega}\subseteq \overline{\Omega}$.  
	For every $\omega \in \tilde{\mathcal{B}} \times \mathcal{B}$, we have
	\[
	\Phi^{t}_{\omega} \left\{ b_i(\omega^-) \right\}_{1 \leq i \leq p} 
	= \left\{ b_i(P_t(\omega^-,\omega^+)) \right\}_{1 \leq i \leq p}.
	\]
	Thus, the map 
	\begin{align}\label{RRAS}
		t \mapsto \mathcal{R}\!\left(P_{t}(\omega^-,\omega^+)\right)
	\end{align}
	is continuous.  
	Moreover, by the definition of $\overline{\Omega}$ on the dense set 
	\[
	\{t_i : i \in \mathbb{N}\} 
	= \bigl\{ \tfrac{k}{2^j} : k \in \mathbb{Z}, \, j \in \mathbb{N} \bigr\} \subseteq \mathbb{R},
	\]
	we have $\mathcal{R}(P_{t_i}(\omega^-,\omega^+)) \leq C $.
	Together with the continuity of the map \eqref{RRAS}, this implies that 
	\(\overline{\Omega} \subseteq \tilde{\Omega} \) 
	yielding that $\tilde{\Omega} $ has full measure.  
	
\end{remark}
We now define a weak type of synchronization in our non-Markovian setting. 
\begin{definition}\label{DEF:WEAK}
	We say that {weak synchronization} holds if there exists a random point \(b: \tilde{\mathcal{B}} \to \mathbb{R}^d\) such that for \(\mathbf{P}\)-almost every \(\omega^- \in \tilde{\mathcal{B}}\),
	\begin{align*}
		\mathrm{supp}(\mu_{\omega^-}) = \{ b(\omega^-) \}.
	\end{align*}
\end{definition}
\begin{remark}\label{KIIASSe}
	In \cite[Definition~2.16]{FGS16a}, weak synchronization means that the minimal point attractor is a singleton. As already stated in Remark~\ref{ARATSasd}, for strongly mixing white-noise RDS, the random measures $\mu_{\omega^-}$ are minimal point attractors.~In our case, we only know that $\text{supp}(\mu_{\omega^-})\subseteq \overline{\cA}(\omega^-)$, which motivates our definition.~Consequently, when weak synchronization in the sense of \cite[Definition~2.16]{FGS16a} holds, i.e.~$\overline{\cA}(\omega^-)=b(\omega^-)$, this implies that each random measure is a Dirac measure sitting on the equilibrium, i.e. $\mu_{\omega^-}=\delta_{b(\omega^-)}$. 
\end{remark} 
The following result provides a necessary and sufficient condition to verify weak synchronization. 
\begin{lemma}\label{SYNCH}
	Let \( \eta > 0 \) and suppose that 
	\begin{align}\label{ASA85sd}
		\lim_{n \to \infty} \chi_{B(0,\eta)}\!\left(\mathcal{R}\!\left(P_n(\omega^-, \omega^+)\right)\right) = 1, 
		\quad \mathbb{P}\text{-a.s.}.
	\end{align}
	Then
	\[
	\mathbf{P}\big(\mathcal{R}(\omega^-) < \eta\big) = 1.
	\]
	In particular, weak synchronization follows if \eqref{ASA85sd} holds for every \( \eta > 0 \).
\end{lemma}
\proof
From the assumption and Fatou's lemma, we have
\begin{align*}
	1 &= \mathbb{E} \left( \liminf_{n \to \infty} \chi_{B(0,\eta)}\left(\mathcal{R}\left(P_n(\omega^-, \omega^+)\right)\right) \right)\leq \liminf_{n \rightarrow \infty} \mathbb{E} \left( \chi_{B(0,\eta)}\left(\mathcal{R}\left(P_n(\omega^-, \omega^+)\right)\right) \right) \\
	&= \mathbb{P} \left( \mathcal{R}\left(P_n(\omega^-, \omega^+)\right) < \eta \right) = \mathbf{P} \left( \mathcal{R}(\omega^-) < \eta \right),
\end{align*}
where in the last step we used the fact that the family \( (\theta_t)_{t \geq 0} \) 
is \( \mathbb{P}\)-invariant (see Theorem \ref{RDS_SDS}). 
Now let $\mathcal{B}_\eta$ denote the set of $\mathbf{P}$-full measure such that for each $\omega \in \mathcal{B}_\eta$ one has $\mathcal{R}(\omega^-) \leq \eta$. It then follows that 
\[
\mathbf{P}\left(\bigcap_{m \geq 1} \mathcal{B}_{\frac{1}{m}}\right) = 1,
\]
and for each $\omega^-$ in this intersection, we have $\mathcal{R}(\omega^-) = 0$. This yields weak synchronization.
\qed
Therefore, it suffices to establish \eqref{ASA85sd} in order to prove weak synchronization. The main idea is to show that if the total mass of the projected measure $\mu$ in $\mathbb{R}^d$ is not concentrated in a certain region of $\R^d$ and is sufficiently spread out, then the size of the attractor (which is bounded by a deterministic constant) will eventually become zero in the long run, according to  Lemma \ref{SYNCH}.

The condition that the density is not concentrated in a certain region can be achieved by tuning the noise intensity, which will be considered in  Section~\ref{sec:ex}. 
\begin{definition}
	For each $\omega=(\omega^-,\omega^+)\in \tilde{\Omega}$, $n\in\N$ and each Borel set $A\subseteq \mathbb{R}^d$, we define
	\begin{align*}
		\mathcal{U}(n,\omega,A)
		:=
		\frac{1}{p}\sum_{1\leq i\leq p}
		\int_{0}^{n}
		\delta_{\{ b_{i}(P_{t}(\omega^-,\omega^+)) \}}(A)\,\mathrm{d}t.
	\end{align*}
	
	We also define the marginal of $\mu$ on $\mathbb{R}^d$ by
	\begin{align}\label{OLKasefzs}
		\pi(A):=[(\Pi_{\R^d})_{\star}(\mu)](A).  
	\end{align}
\end{definition}

\begin{lemma}\label{BIRK}
	For each Borel set $A\subseteq \mathbb{R}^d$, it holds that
	\begin{align*}
		\lim_{n\rightarrow \infty}\frac{1}{n}\mathcal{U}(n,\omega,A)
		=
		\pi(A)
		\qquad \mathbb{P}\text{-a.s.}
	\end{align*}
\end{lemma}
\proof
By definition, for $n,m\in\N$,
\begin{align*}
	\mathcal{U}(n+m,\omega,A)
	=
	\mathcal{U}(m,\theta_{n}\omega,A)
	+
	\mathcal{U}(n,\omega,A).
\end{align*}

Moreover, using the invariance of $\theta_t$ under the measure $\mathbb{P}$ for each $t\geq 0$, we obtain
\begin{align*}
	\frac{1}{p}\sum_{1\leq i\leq p}
	\mathbb{P}\Bigl(
	\omega\in\tilde{\Omega}:
	b_{i}(P_{t}(\omega^-,\omega^+))\in A
	\Bigr)
	&=
	\frac{1}{p}\sum_{1\leq i\leq p}
	\mathbf{P}\Bigl(
	\omega^-\in\tilde{\mathcal{B}}:
	b_{i}(\omega^-)\in A
	\Bigr) \\
	&=
	\pi(A).
\end{align*}

Consequently, by Birkhoff's ergodic theorem and the fact that $\theta_1:\Omega\to\Omega$ is ergodic, we conclude that 
\begin{align*}
	\lim_{n\rightarrow \infty}\frac{1}{n}\mathcal{U}(n,\omega,A)
	=
	\mathbb{E}\bigl(\mathcal{U}(1,\omega,A)\bigr)
	=
	\pi(A).
\end{align*}
\qed
In the following, we aim to quantify the Lebesgue measure of the trajectories of~\eqref{MAIN} belonging to contracting / expanding regions in $\R^d$.
\begin{definition}\label{MNaseea}
	We set
	\begin{align*}
		A_{1}
		&:= \overline{B(0,R+C)}, \\
		A_{2}
		&:= \mathbb{R}^{d}\setminus \overline{B(0,R+C)}.
	\end{align*}
	
	For each $\omega \in \tilde{\Omega}$ and $n\in\N$, we define
	\begin{align*}
		\mathcal{M}(n,\omega)
		&:=
		\left\{
		t\in [0,n]:
		b_{i}(P_{t}(\omega^-,\omega^+))\in A_1
		\text{ for all } 1\leq i \leq p
		\right\},
		\\
		\mathcal{N}(n,\omega)
		&:=
		[0,n]\setminus \mathcal{M}(n,\omega)
		\\
		&=
		\left\{
		t\in [0,n]:
		\text{there exists } 1\leq i \leq p
		\text{ such that }
		b_{i}(P_{t}(\omega^-,\omega^+))\in A_2
		\right\}.
	\end{align*}
\end{definition}

The following lemma provides a bound for $\mathcal{R}\!\left(P_{n}(\omega^-,\omega^+)\right)$.

\begin{lemma}\label{ESTSTSY} 
	For each $\omega \in \tilde{\Omega}$ and $n\in\N$, we have
	\begin{align*}
		\mathcal{R}\!\left(P_{n}(\omega^-,\omega^+)\right) 
		&\leq
		\mathcal{R}\!\left(\omega^-\right)
		\exp\!\left(
		- C_{4}^{F}\,\mu_{\mathrm{Leb}}(\mathcal{N}(n,\omega))
		+ C_{3}^{F}\,\mu_{\mathrm{Leb}}(\mathcal{M}(n,\omega))
		\right).
	\end{align*}
\end{lemma}
\proof 
First, recall from Assumption~\ref{AASQw98a1} that for each $\omega=(\omega^-,\omega^+)\in\tilde{\Omega}$ and each $t\geq 0$,
\begin{align}\label{YAsjassssdss}
	\begin{split}
		\left(\Phi^{t}_{\omega} \bigl( b_i(\omega^-) \bigr)\right)_{1 \leq i \leq p}
		&=
		\bigl( b_i(P_t(\omega^-,\omega^+)) \bigr)_{1 \leq i \leq p}, \\
		\mathcal{R}(\omega^-)
		&\leq C.
	\end{split}
\end{align}
We define
\begin{align*}
	\alpha(t,\omega)
	:=
	-C_4^F \mathbf{1}_{\mathcal{N}(n,\omega)}(t)
	+
	C_3^F  \mathbf{1}_{\mathcal{M}(n,\omega)}(t).
\end{align*}

Obviously, $\alpha(t,\cdot)$ is measurable thanks to the continuity of the trajectories
\begin{align*}
	t \mapsto \Phi^{t}_{\omega}\bigl(b_{i}(\omega^-)\bigr), \qquad i=1,\dots,p.
\end{align*}

Let $\{t_{k,m}\}_{0\leq k\leq m}$ be a family of partitions of $[0,n]$ such that
\[
t_{0,m}=0<t_{1,m}<\cdots<t_{m-1,m}<t_{m,m}=n,
\]
with mesh size converging to zero as $m\to\infty$. Then we define a sequence of {measurable} functions
\begin{align}\label{JNass}
	\alpha^{m}(t,\omega)
	:=
	\sum_{k=0}^{m-2} c_{k,m}\,\mathbf{1}_{[t_{k,m},t_{k+1,m})}(t)
	+
	c_{m-1,m}\,\mathbf{1}_{[t_{m-1,m},t_{m,m}]}(t),
\end{align}
where, for $0\leq k\leq m-2$, we set
\begin{align*}
	c_{k,m}
	:=
	\begin{cases}
		-C_4^F, & \text{if } [t_{k,m},t_{k+1,m}) \subseteq \mathcal{N}(n,\omega), \\
		C_3^F, & \text{otherwise.}
	\end{cases}
\end{align*}

For $k=m-1$, we define $c_{m-1,m}$ analogously using the interval $[t_{m-1,m},t_{m,m}]$.
We prove that
\begin{align}\label{BNAsssss}
	\forall t\in [0,n]: \quad \alpha^{m}(t,\omega)\rightarrow \alpha(t,\omega)
	\quad \text{as } m\to \infty.
\end{align}
To this end, we consider two cases:
\begin{enumerate}
	\item If $t\in \mathcal{M}(n,\omega)$, then it is clear that for every $m$
	\begin{align*}
		\alpha^{m}(t,\omega)=\alpha(t,\omega)=C_{3}^{F}.
	\end{align*}
	\item If $t\in \mathcal{N}(n,\omega)$, then by the definition of $\mathcal{N}(n,\omega)$ and \eqref{YAsjassssdss}, there exists $1\leq i\leq p$ such that
	$
	\Phi^{t}_{\omega}\bigl(b_{i}(\omega^-)\bigr)\in A_2.$
	Since $A_2$ is open and the trajectory
	$s \mapsto \Phi^{s}_{\omega}(b_{i}(\omega^-))$
	is continuous, we deduce that there exists an interval $(a,b)$ with $t\in (a,b)$ such that
	\begin{align}
		\forall s\in [a,b]: \quad \Phi^{s}_{\omega}\bigl(b_{i}(\omega^-)\bigr)\in A_2.
	\end{align}
	Hence,
	$[a,b]\subseteq \mathcal{N}(n,\omega).$
	Since the mesh size of the partition $\{t_{k,m}\}_{0\leq k\leq m}$ converges to zero, there exists $M$ such that for all $m\geq M$, if for $K_t\in \lbrace 0,1,...,m-1 \rbrace$
	where
	$t\in [t_{K_t,m}, t_{K_t+1,m})$,
	(if $K_t=m-1$, we consider $[t_{m-1,m},t_{m,m}]$), then
	$
	t \in (a,b).
	$
	Therefore, since $[a,b]\subseteq \mathcal{N}(n,\omega)$, it follows that for all $m\geq M$,
	\begin{align*}
		\alpha^{m}(t,\omega)=\alpha(t,\omega)=-C_{4}^{F}.
	\end{align*}
\end{enumerate}
Now, for each partition $\{t_{k,m}\}_{0\leq k\leq m}$  and each $0\leq k<m$, we distinguish the following two cases:
\begin{enumerate}
	\item If $[t_{k,m},t_{k+1,m})\subseteq \mathcal{N}(n,\omega)$, then for each $t\in [t_{k,m},t_{k+1,m})$, it follows that there exists $j\in\{1,\ldots,p\}$ such that
	$b_{j}(P_{t}(\omega^-,\omega^+))\in A_2$. Since $\mathcal{R}(\omega^-)\leq C$
	we get for a fixed $i\in\{1,\ldots,p\}$ that
	\begin{align*}
		|b_i(P_t(\omega^-,\omega^+))| &\geq |b_j(P_t(\omega^-,\omega^+))| - |b_i(P_t(\omega^-,\omega^+))-b_j(P_t(\omega^-,\omega^+))|\\
		& \geq R+C- \cR(\omega^-) > R.
	\end{align*}
	In conclusion, 
	\begin{align}\label{Nbaww}
		|b_{i}(P_{t}(\omega^-,\omega^+))| \geq R,
		\qquad \forall t\in [t_{k,m},t_{k+1,m}), \ \forall\, 1\leq i\leq p. 
	\end{align}
	Thanks to~\eqref{monotone st11}, for all $|x|,|y|\geq R$ we have
	\begin{align*}
		\frac{\mathrm{d}}{\mathrm{d} t}\left| \Phi^{t}_{\omega}(x) - \Phi^{t}_{\omega}(y) \right|^2 
		&= 2\left\langle F(\Phi^{t}_{\omega}(x)) - F(\Phi^{t}_{\omega}(y)), \, \Phi^{t}_{\omega}(x) - \Phi^{t}_{\omega}(y) \right\rangle \\
		&\leq -2C^F_4 \left| \Phi^{t}_{\omega}(x) - \Phi^{t}_{\omega}(y) \right|^2.
	\end{align*}
	Keeping in mind~\eqref{YAsjassssdss}, this leads to	\begin{align*}
		\mathcal{R}\!\left(P_{t_{k+1,m}}(\omega^-,\omega^+)\right) 
		&\leq
		\mathcal{R}\!\left(P_{t_{k,m}}(\omega^-,\omega^+)\right)
		\exp\!\left(- C_{4}^{F} (t_{k+1,m}-t_{k,m})\right).
	\end{align*}
	
	\item If $[t_{k,m},t_{k+1,m})\not\subseteq \mathcal{N}(n,\omega)$, then we use Assumption~\ref{Drift} and argue analogously to the previous case to obtain
	\begin{align*}
		\mathcal{R}\!\left(P_{t_{k+1,m}}(\omega^-,\omega^+)\right) 
		&\leq
		\mathcal{R}\!\left(P_{t_{k,m}}(\omega^-,\omega^+)\right)
		\exp\!\left(C_{3}^{F} (t_{k+1,m}-t_{k,m})\right).
	\end{align*}
	
\end{enumerate}
Consequently, by iterating this procedure on each interval $[t_{k,m},t_{k+1,m}]$ and using~\eqref{JNass}, we obtain
\begin{align*}
	\mathcal{R}\!\left(P_{n}(\omega^-,\omega^+)\right) 
	\leq
	\mathcal{R}\!\left(\omega^-\right)
	\exp\!\left(\int_{0}^{n} \alpha^{m}(t,\omega)\,\mathrm{d}t\right).
\end{align*}

Thus, thanks to~\eqref{BNAsssss} and using the dominated convergence theorem, we infer that
\begin{align*}
	\mathcal{R}\!\left(P_{n}(\omega^-,\omega^+)\right) 
	\leq
	\mathcal{R}\!\left(\omega^-\right)
	\exp\!\left(\int_{0}^{n} \alpha(t,\omega)\,\mathrm{d}t\right).
\end{align*}

which completes the proof.
\qed
\begin{remark}
	The argument at first glance might suggest that one can follow the same procedure to obtain a similar bound on $\|\Phi^{n}_{\omega}(x)-\Phi^{n}_{\omega}(y)\|$. However, in the proof (see~\eqref{Nbaww}), 
	we used the fact that for each $\omega^-$ one has $\mathcal{R}(\omega^-)\leq C$. Thus, even if the trajectories initially satisfy the required bound, this property cannot be guaranteed for large $n$.
\end{remark}
\begin{remark}
	Let $\Gamma\subset \mathbb{R}^d$ be a Borel set with Lebesgue measure zero.~Then, by Lemma \ref{BIRK} 
	\begin{align*}
		\lim_{n\to\infty}\frac{1}{n}\mathcal{U}(n,\omega,\Gamma)
		=\pi(\Gamma)=0,
		\qquad
		\mathbb{P}\text{-a.s.},
	\end{align*}
	where the last equality follows from~\cite[Theorem 1.1]{LPS23}.
	In particular, this implies that the trajectory
	$t \mapsto \Phi^{t}_{\omega}\bigl(b_i(\omega^-)\bigr)$, for  $1 \leq i \leq p$, spends a negligible amount of time on the boundary $\partial A_1$ in an averaged sense.
\end{remark}
The next lemma provides an average growth estimate for the Lebesgue measures of the sets
$\mathcal{M}(n,\omega)$ and
$\mathcal{N}(n,\omega)$.

\begin{lemma}\label{BNassssaq}
	For $\mathbb{P}$-almost every $\omega\in\tilde{\Omega}$, we have
	\begin{align*}
		\limsup_{n\rightarrow \infty}
		\frac{\mu_{\mathrm{Leb}}(\mathcal{M}(n,\omega))}{n}
		\leq
		\pi(A_1),
	\end{align*}
	and
	\begin{align*}
		\liminf_{n\rightarrow \infty}
		\frac{\mu_{\mathrm{Leb}}(\mathcal{N}(n,\omega))}{n}
		\geq
		\pi(A_2)=1-\pi(A_1).
	\end{align*}
\end{lemma}
\proof
From the definition of $\cM$, we have for $t\in[0,n]$ that 
\[ \chi_{\mathcal{M}(n,\omega)}(t) \leq \frac{1}{p} \sum\limits_{i=1}^p \chi_{A_1}(b_i(P_t(\omega^-,\omega^+))) \]
and therefore
\begin{align*}
	\mu_{\mathrm{Leb}}(\mathcal{M}(n,\omega))
	\leq
	\frac{1}{p}\sum_{i=1}^{p}
	\int_{0}^{n}
	\delta_{\{ b_{i}(P_{t}(\omega^-,\omega^+)) \}}(A_1)\,\mathrm{d}t.
\end{align*}
Similarly
\[ \frac{1}{p} \sum\limits_{i=1}^p \chi_{A_2} (b_i(P_t(\omega^-,\omega^+))) \leq \chi_{\mathcal{N}(n,\omega)}(t) ,\]
which leads to 
\begin{align*}
	\frac{1}{p}\sum_{i=1}^{p}
	\int_{0}^{n}
	\delta_{\{ b_{i}(P_{t}(\omega^-,\omega^+)) \}}(A_2)\,\mathrm{d}t
	\leq
	\mu_{\mathrm{Leb}}(\mathcal{N}(n,\omega)).
\end{align*}
The claim now follows from Lemma~\ref{BIRK}.
\qed
We can now state the main result of this section.
\begin{theorem}\label{WESYN}
	Suppose that
	\begin{align}\label{AHSyss}
		\pi(A_1)
		<
		\frac{C_{4}^{F}}{C_{3}^{F}+C_{4}^{F}}.
	\end{align}
	Then, under Assumptions \ref{AASQw98a} and \ref{AASQw98a1}, for every $\eta > 0$,
	\[
	\mathbf{P}\!\left(\mathcal{R}(\omega^-) < \eta \right) = 1.
	\]
	This implies weak synchronization for \eqref{MAIN}.
\end{theorem}
\proof
By Lemma~\ref{ESTSTSY}, we have
\begin{align}\label{NBgta}
	\mathcal{R}\!\left(P_{n}(\omega^-,\omega^+)\right) 
	&\leq
	\mathcal{R}\!\left(\omega^-\right)
	\exp\!\left(
	n\left(
	- C_{4}^{F}\,\frac{\mu_{\mathrm{Leb}}(\mathcal{N}(n,\omega))}{n}
	+ C_{3}^{F}\,\frac{\mu_{\mathrm{Leb}}(\mathcal{M}(n,\omega))}{n}
	\right)\right).
\end{align}
Thus, it follows from Lemma~\ref{BNassssaq}
\begin{align}\label{BVasss}
	\limsup_{n\to\infty} \Big( -C^{F}_4 \frac{\mu(N)
	}{n} + C^{F}_3 \frac{\mu(M)}{n} \Big)\leq -C^{F}_4 (1-\pi(A_1)) + C^{F}_3 \pi(A_1)<0
\end{align}
Since \(\mathcal{R}\) is a non-negative random variable, it follows from \eqref{NBgta} and \eqref{BVasss} that
\[
\lim_{n\to\infty}
\mathcal{R}\!\left(P_{n}(\omega^-,\omega^+)\right)
=0,
\qquad \mathbb{P}\text{-a.s.}
\]
This immediately implies \eqref{ASA85sd}. Therefore, the claim follows from Lemma~\ref{SYNCH}.

\qed

\section{Applications}\label{sec:ex}
We now show how to verify weak synchronization for the SDE~\eqref{MAIN}.~As already stated, the first step in this direction is given by the negativity of the top Lyapunov exponent.~This holds true once the diffusion coefficient $\sigma$ is sufficiently large, as established in~\cite{BNGH26}. 
\begin{definition}
For $\kappa > 0$ and $\theta \geq 1$, define
\[
T_{\theta,\kappa}
:= \left\{
\sigma \in GL(d,\mathbb{R})
\;\middle|\;
\|\sigma\| \ge \kappa,\;
\|\sigma\|\,\|\sigma^{-1}\| \leq \theta
\right\},
\]
where $\|\cdot\|$ denotes the standard operator norm induced by the Euclidean norm on $\mathbb{R}^d$ and $GL(d,\mathbb{R})$ is the space of invertible $d \times d$ matrices.
\end{definition}
Obviously, for $\kappa_1 > \kappa_2$ we have the inclusion
\begin{align}\label{KLassd}
T_{\theta, \kappa_1} \subset T_{\theta, \kappa_2}.
\end{align}
\begin{proposition}\label{EXASA}
Consider the SDE~\eqref{MAIN} under Assumptions~\ref{Drift}, together with the additional requirement that condition~4) is satisfied for every $H \in (0,1)$. Moreover, suppose that Assumption~\ref{AASQw98a} holds. Then, for every $\theta \geq 1$,  there exists a constant $\kappa_\theta > 0$ such that for each $\sigma \in T_{\theta,\kappa_\theta}$, the top Lyapunov exponent $\lambda^\sigma_1$ satisfies 
\begin{align*}
	\lambda^{\sigma}_1 < 0.
\end{align*}
Furthermore, assuming that $\kappa_\theta$ is sufficiently large implies \eqref{AHSyss}.
\end{proposition}

\proof

For the negativity of the top Lyapunov exponent, we refer to \cite[Theorem~5.19]{BNGH26}.~Here we only verify \eqref{AHSyss} which follows by similar arguments as \cite[Theorem~5.19]{BNGH26}. For the convenience of the reader, we sketch the main ideas.\\

The goal is to show that, for a fixed $\theta$ and $\sigma \in T_{\theta,\kappa}$, where $\kappa$ is sufficiently large, the total mass of $\pi=\pi^{\sigma}$ (recall~\eqref{OLKasefzs}) on $B(0,R+C)$ becomes arbitrarily small. To this end, we rescale the SDE~\eqref{MAIN} as follows
\begin{align}\label{RESCA}
\begin{cases}
	\txtd Z_t^\sigma = \|\sigma\|^{-1} F(\|\sigma\| Z_t^\sigma) \, \txtd t + \|\sigma\|^{-1} \sigma \, \txtd B_t^H, \\
	Z_0^\sigma=x \in \mathbb{R}^d.
\end{cases}
\end{align}
Let
\( (Z_{t,x}^\sigma)_{t \geq 0} \) denote its solution. Then, for every \( t \geq 0 \) and \( x \in \mathbb{R}^d \), it is easy to see that
\begin{align*}
Z_{t,x}^{\sigma} = \|\sigma\|^{-1} Y_{t,\|\sigma\| x}^{\sigma},
\end{align*}
where $Y^\sigma_{t,x}$ denotes the solution of~\eqref{MAIN} starting from $x\in\R^d$.
Let
\begin{align}\label{IK63asd}
F^{\sigma}(\xi) := \|\sigma\|^{-1} F(\|\sigma\|\xi).
\end{align}
Then for every \(\xi_1,\xi_2 \in \mathbb{R}^d\) and for all \(\|\sigma\|\ge 1\), we have
\begin{align}
\begin{split}
	\langle F^{\sigma}(\xi_2) - F^{\sigma}(\xi_1), \xi_2-\xi_1 \rangle
	&= \|\sigma\|^{-2}
	\left\langle
	F(\|\sigma\|\xi_2)-F(\|\sigma\|\xi_1),
	\|\sigma\|(\xi_2-\xi_1)
	\right\rangle \\
	&\leq \min \left\{
	\|\sigma\|^{-2} C_1^F - C_2^F |\xi_2-\xi_1|^2,\;
	C_3^F |\xi_2-\xi_1|^2
	\right\} \\
	&\leq \min \left\{
	C_1^F - C_2^F |\xi_2-\xi_1|^2,\;
	C_3^F |\xi_2-\xi_1|^2
	\right\}.
\end{split}
\end{align}

Since \(DF\) is globally bounded, there exists \(\overline{C}_F>0\) such that for all \(\|\sigma\|\ge 1\),
\begin{align}
|F^{\sigma}(\xi)| \leq \overline{C}_F (1 + |\xi|),
\qquad
\|D_\xi F^{\sigma}(\xi)\| \leq \overline{C}_F,
\quad \text{for all } \xi \in \mathbb{R}^d.
\end{align}
Finally, for \(|\xi_1|,|\xi_2| \ge \frac{R}{\|\sigma\|}\),
\begin{align}
\langle F^{\sigma}(\xi_2) - F^{\sigma}(\xi_1), \xi_2-\xi_1 \rangle
\leq -C_4^F |\xi_2-\xi_1|^2.
\end{align}
Thus, thanks to Theorem~\ref{RDS_SDS_MEAU}, the SDE~\eqref{RESCA} admits a unique invariant measure $\tilde{\mu}^{\sigma}$. Let $\tilde{\pi}^{\sigma}$ be the projected measure associated with $\tilde{\mu}^{\sigma}$ as defined in~\eqref{OLKasefzs}. By \cite[Lemma~5.8]{BNGH26}, we have
\begin{align}\label{UAISd96a}
\pi^{\sigma}\bigl(B(0,R+C)\bigr)
=
\tilde{\pi}^{\sigma}\bigl(\|\sigma\|^{-1}B(0,R+C)\bigr).
\end{align}
Furthermore, by \cite[Proposition~5.17]{BNGH26}, the measure $\tilde{\pi}^{\sigma}$ admits a density $\tilde{p}^{\sigma}_{\infty}$ with respect to the Lebesgue measure. Hence,
\begin{align}\label{OLLAs63a}
\pi^{\sigma}\bigl(B(0, R+C)\bigr)
=
\int_{\|\sigma\|^{-1} B(0,R+C)} \tilde{p}^{\sigma}_{\infty}(y)\,\mathrm{d}y.
\end{align}
Moreover, by \cite[Proposition~5.18]{BNGH26}, there exist constants \(C_1,C_2>0\), depending only on \(\theta\), \(H\), \(\kappa\), and \(R+C\), such that
\begin{align}
\sup_{\sigma \in T_{\theta,\kappa}}
\tilde{p}^{\sigma}_{\infty}(y)
\leq
C_1 \exp(-C_2 |y|^2),
\qquad \text{for all } y \in B(0,R+C).
\end{align}
Note that this bound is uniform in \(\sigma \in T_{\theta,\kappa}\). Keeping \eqref{KLassd} in mind, choosing \(\kappa\) sufficiently large, the term in~\eqref{OLLAs63a}
can be made arbitrarily small for each \(\sigma \in T_{\theta,\kappa}\). Therefore, the condition~\eqref{AHSyss} is satisfied in this regime.
\qed
In particular, this yields the following result.
\begin{corollary}\label{KMNas}
Let $\lambda^{\sigma}_{1} < 0$, and suppose that \eqref{AHSyss} holds. Moreover, assume that the hypotheses of Theorem~\ref{ASASiid} are satisfied. 
Then weak synchronization holds for~\eqref{MAIN}.
\end{corollary}

\proof	Thanks to Corollary~\ref{ATRTASs} and Remark~\ref{IAKSOded}, 
one can verify Assumption~\ref{AASQw98a1}. The claim is derived from Theorem~\ref{WESYN}.
\qed

\begin{example}
We provide an example of a drift term $F$ that satisfies the conditions of Proposition \ref{EXASA} and Corollary \ref{KMNas}. To this aim we consider a $C^{2}$-function $F$ defined for some $0 < R_{1} < R_{2}$ as
\begin{align*}
	F(x)=\begin{cases} &
		x - x|x|^{2}, ~~~~~  |x| < R_{1}, \\
		& H(x) - x, ~~~~~ |x| > R_{2},  
	\end{cases}
\end{align*}
where $H : \mathbb{R}^{d} \to \mathbb{R}^{d}$ is a $C^{2}$ function with bounded derivatives satisfying $\|DH\| < 1$. 
The function $F$ can then be defined on the annulus $R_{1} \le |x| \le R_{2}$ in such a way that $F$ is globally of class $C^{2}$. 
\end{example}
\section{Conclusion and Outlook}\label{outlook} 
In this work we showed weak synchronization for SDEs driven by additive fractional Brownian motion, where the drift is eventually strictly monotone and the noise intensity is large.~We conclude by highlighting several open problems that we intend to address in forthcoming works.

\begin{itemize}
\item (Relaxing Assumption \ref{AHSyss}) The property
\begin{align}\label{cond}
	\lim_{n\to\infty}
	\mathcal{R}\!\left(P_{n}(\omega^-,\omega^+)\right)
	=0,
	\qquad \mathbb{P}\text{-a.s.}
\end{align}
is expected to hold without assuming~\eqref{AHSyss}. To this end, we briefly outline how one can prove that
\begin{align}\label{BNassss}
	\liminf_{n\to\infty}
	\mathcal{R}\!\left(P_{n}(\omega^-,\omega^+)\right)
	=0,
	\qquad \mathbb{P}\text{-a.s.}
\end{align} 
The main idea is to relate the dynamics of~\eqref{MAIN} to the dynamics of the ODE
\[
\mathrm{d}Y_t = F(Y_t)\,\mathrm{d}t + \sigma\,\mathrm{d}h^v(t), \qquad Y_0 = x \in \mathbb{R}^d,
\]
where 
$h^v(t) = v t \sigma^{-1} e_1
$
is a deterministic control. 
For \(v\) sufficiently large, the control dominates the drift \(F\), and the solution spends most of the time outside the ball \(\overline{B(0,R+C)}\). In this region, the monotonicity condition implies exponential contraction of trajectories.~By the support theorem (e.g.~\cite[Prop.~5.8]{HO07}), there exists a set of noise realizations of positive probability such that the SDE~\eqref{MAIN} follows the  dynamics of the controlled ODE, so that solutions remain mostly outside \(\overline{B(0,R+C)}\) and hence become closer.
Finally, using the invariance relation
\[
\Phi^{t}_{\omega}\{b_i(\omega^-)\}_{i=1}^p
=
\{b_i(P_t(\omega^-,\omega^+))\}_{i=1}^p,
\]
together with Poincar\'e's recurrence, such favorable regimes occur infinitely often along typical trajectories which yields \eqref{BNassss}.
However, verifying~\eqref{cond} without assuming~\eqref{AHSyss} would require additional probabilistic and ergodic arguments.
\item (Small noise regime) From Proposition \ref{EXASA} and Corollary \ref{KMNas}, it follows that both the drift and diffusion coefficients contribute to synchronization. 
Furthermore, based on an interplay between drift and diffusion, 
there might exist positive Lyapunov exponents, which would imply a chaotic behavior of the underlying system.~This situation will be analyzed in a small noise regime, in contrast to the situation explored here.
\item (Negativity of top Lyapunov exponent) For the SDE \eqref{MAIN} in dimension one, weak synchronization in the sense of \cite{FGS17} holds, meaning that the weak point attractor $\cA(\omega^-)$ is a singleton. We expect that the top Lyapunov exponent is negative in this case.
\item (Speed of convergence towards the singleton) For small values of $H$, the noise exhibits more rapid fluctuations.~Intuitively, these fluctuations may facilitate transitions of trajectories from the expansive region $B(0,R)$ to the contracting region $\mathbb{R}^d \setminus \overline{B(0,R)}$. As a result, trajectories observed over a sufficiently long time interval $[0,T]$ should synchronize faster than for large values of the Hurst parameter. In this case, we expect that trajectories spend more time in the expansive region.~Moreover, the drift itself tends to drive trajectories from the contracting region back into the expansive region.~Hence, the trajectories that enter the expansive region are likely to remain there for longer periods. Consequently, the synchronization rate is expected to be slower.
\begin{figure}[h]
	\centering
	\begin{tikzpicture}[scale=.85]
		\tikzset{
			region/.style={gray!15},
			traj_small/.style={red, very thick},
			traj_large/.style={red, very thick},
			arrowmark/.style={postaction={decorate},
				decoration={markings,
					mark=at position 0.15 with {\arrow{>}},
					mark=at position 0.30 with {\arrow{>}},
					mark=at position 0.45 with {\arrow{>}},
					mark=at position 0.60 with {\arrow{>}},
					mark=at position 0.75 with {\arrow{>}},
					mark=at position 0.90 with {\arrow{>}}
				}
			}
		}
		
		
		\node[font=\large] at (-5,3) {Small $H$};
		
		\fill[region] (-5,0) circle (2);
		\draw[thick] (-5,0) circle (2);
		
		\node at (-5,-2.8) {Contracting region};
		
		\draw[traj_small, arrowmark]
		plot[smooth,tension=0.9] coordinates {
			(-7.2,1)
			(-3,2.9)
			(-6.8,-2.5)
			(-5.2,2.5)
			(-4.2,-2.3)
			(-2,2.6)
			(-5.0,-2.3)
			(-4,1.5)
			(-7.0,-1.0)
		};
		
		
		\node[font=\large] at (5,3) {Large $H$};
		
		\fill[region] (5,0) circle (2);
		\draw[thick] (5,0) circle (2);
		
		\node at (5,-2.8) {Contracting region};
		
		\draw[traj_large, arrowmark]
		plot[smooth,tension=1.2] coordinates {
			(1.8,1.6)
			(2.8,1.2)
			(3.8,1.8)
			(4.8,2.1)
			(6.0,1.4)
			(6.8,0.3)
			(6.1,-0.9)
			(5.0,-1.8)
			(3.8,-1.3)
			(2.8,-0.3)
		};
		
		\fill[red] (1.8,1.6) circle (1.5pt);
		\fill[red] (2.8,-0.3) circle (1.5pt);
		
	\end{tikzpicture}
	\caption{The trajectories of \eqref{MAIN} for small and large values of the Hurst parameter.} 
\end{figure}
\end{itemize}

\appendix
\section{A measure theoretical result}\label{appendix}
The following result allows us to construct a measurable selection map which is required for the proof of Corollary \ref{STABSS}.
\begin{lemma}\label{MESUA}
Assume that \( \mathcal{Z} \) is a Polish space endowed with its Borel \( \sigma \)-algebra, and let \( A \subseteq \mathbb{R}^d \times \mathcal{Z} \) be a compact set. For each \( z \in \mathcal{Z} \), define the fiber
\[
A(z) := \{ y \in \mathbb{R}^d \mid (y, z) \in A \}.
\]
Then, there exists a Borel-measurable function
\[
K : \Pi_{\mathcal{Z}}(A) \longrightarrow \mathbb{R}^d,
\]
such that \( K(z) \in A(z) \) for every \( z \in \Pi_{\mathcal{Z}}(A) \). Moreover, this map admits a Borel-measurable extension \( \tilde{K} : \mathcal{Z} \to \mathbb{R}^d \).
\end{lemma}
\proof
Note that since \( A \) is a compact subset of the Polish space \( \mathbb{R}^d \times \mathcal{Z} \), it is sequentially compact, so  \( \Pi_{\mathcal{Z}}(A)\) is closed. Also, for each \( z \in \mathcal{Z} \), the fiber
\[
A(z) := \{ y \in \mathbb{R}^d \mid (y, z) \in A \}
\]
is a compact subset of \( \mathbb{R}^d \). We claim that the family of compact sets \( \{ A(z) \}_{z \in \mathcal{Z}} \) is weakly measurable, meaning that for every open set \( U \subseteq \mathbb{R}^d \), the set
\[
\{ z \in \mathcal{Z} \mid A(z) \cap U \neq \emptyset \} = \Pi_{\mathcal{Z}} \big( A \cap ( U \times\mathcal{Z} ) \big)
\]
belongs to \( \sigma(\mathcal{Z}) \). 
To verify this, observe that every open set \( U \subseteq \mathbb{R}^d \) can be written as a countable union \( U = \bigcup_{n \geq 1} K_n \), where each \( K_n \) is a compact subset of \( \mathbb{R}^d \). Then,
\[
\Pi_{\mathcal{Z}} \big( A \cap (U \times\mathcal{Z}) \big) = \bigcup_{n \geq 1} \Pi_{\mathcal{Z}} \big( A \cap ( K_n\times\mathcal{Z}) \big).
\]
Each set \( \Pi_{\mathcal{Z}} \big( A \cap ( K_n\times\mathcal{Z}) \big) \) is closed in \( \mathcal{Z} \), which can be shown using a standard sequential compactness argument: a convergent sequence of points from this projection has a limit that must also lie in this set, due to the compactness of both \( A \) and \( K_n \). 
Thus, the weak measurability of the family \( \{A(z)\}_{z \in \mathcal{Z}} \) is established. Moreover, since \( \Pi_{\mathcal{Z}}(A)\) is closed, we have
\[
\sigma\left(\Pi_{\mathcal{Z}}(A)\right) = \sigma(\mathcal{Z}) \cap \Pi_{\mathcal{Z}}(A),
\]
implying that \( \sigma\left(\Pi_{\mathcal{Z}}(A)\right) \subseteq \sigma(\mathcal{Z}) \).
By the Kuratowski-Ryll-Nardzewski measurable selection theorem~\cite[Theorem 18.13]{BC06}, we can therefore construct a Borel-measurable function
\[
\mathcal{K} : \Pi_{\mathcal{Z}}(A) \longrightarrow \mathbb{R}^d
\]
such that \( \mathcal{K}(z) \in A(z) \) for every \( z \in \Pi_{\mathcal{Z}}(A) \). Furthermore, we may extend this function to a Borel-measurable map \( \tilde{\mathcal{K}} : \mathcal{Z} \to \mathbb{R}^d \) by defining
\[
\tilde{\mathcal{K}}(z) := 
\begin{cases}
\mathcal{K}(z), & \text{if } z \in \Pi_{\mathcal{Z}}(A), \\
0, & \text{otherwise}.
\end{cases}
\]
The closedness of \( \Pi_{\mathcal{Z}}(A) \) ensures that \( \tilde{\mathcal{K}} \) is Borel-measurable.
\qed
\bibliographystyle{alpha}
\bibliography{refs}

\end{document}